\documentclass[11pt]{amsart}
\usepackage[a4paper,margin=2.5cm]{geometry}

\usepackage{amsmath,amsthm,amssymb,amsfonts,amscd}
\usepackage{amsxtra}
\usepackage{eucal}
\usepackage{mathrsfs}
\usepackage{color}
\usepackage{xcolor}
\usepackage{hyperref,cleveref}
\usepackage{enumerate}
\usepackage{blkarray}
\usepackage{enumitem}
\usepackage{tcolorbox}
\usepackage{ifthen}
\usepackage{quiver} 

\usepackage{circuitikz}

\usepackage{biblatex}

\usepackage{tikz,tikz-cd}
\usetikzlibrary{3d}
\usepackage{tikz-3dplot}
\usetikzlibrary{positioning}
\usetikzlibrary{decorations.pathreplacing,decorations.markings,decorations.pathmorphing}

\usetikzlibrary{decorations.markings,arrows.meta}
\tikzset{
  midarrow/.style={
    postaction={decorate},
    decoration={markings, mark=at position .57 with {\arrow[draw=blue]{Stealth}}}
  },
}

\tikzstyle{v}=[circle, draw, solid, fill=black, inner sep=0pt, minimum width=3pt]
\tikzstyle{vo}=[circle, draw, solid, fill=white, inner sep=0pt, minimum width=3pt]

\theoremstyle{plain}
\newtheorem{thm}{Theorem}[section]
\newtheorem{lem}[thm]{Lemma}
\newtheorem{cor}[thm]{Corollary}
\newtheorem{prop}[thm]{Proposition}

\theoremstyle{definition}
\newtheorem{defn}[thm]{Definition}
\newtheorem{rmk}[thm]{Remark}
\newtheorem{ex}[thm]{Example}

\newcommand{\bR}{\mathbb{R}}

\newcommand{\bZ}{\mathbb{Z}}

\newcommand{\cA}{\mathcal{A}}
\newcommand{\cB}{\mathcal{B}}

\newcommand{\cG}{\mathcal{G}}

\newcommand{\cT}{\mathcal{T}}

\newcommand{\pP}{\mathsf{P}}
\newcommand{\pQ}{\mathsf{Q}}
\newcommand{\pR}{\mathsf{R}}
\newcommand{\pC}{\mathsf{C}}

\newcommand{\bfe}{\mathbf{e}}
\newcommand{\bfv}{\mathbf{v}}

\newcommand{\symdiff}{\triangle}

\newcommand{\im}{\text{im}}

\newcommand{\alt}{\mathrm{Alt}}

\usepackage{lineno}

\title{A polyhedral approach \linebreak to homotopy theorems in matroid theory}
\author{Changxin Ding}
\address{School of Mathematics, Georgia Institute of Technology, USA}
\email{dcx.math@outlook.com}

\author{Donggyu Kim}
\address{School of Mathematics, Georgia Institute of Technology, USA}
\email{donggyu.math@gmail.com}

\begin{document}

\begin{abstract}
We give a new proof of Maurer's homotopy theorem for matroids using polyhedral methods, in contrast to Maurer's original combinatorial proof. The same polyhedral approach yields a homotopy theorem for delta-matroids, from which the corresponding results for matroids, even delta-matroids, and antisymmetric matroids follow. We further prove an analogous theorem for integral polymatroids.
\end{abstract}

\maketitle

\section{Introduction}\label{sec: intro}

A \emph{matroid} admits several cryptomorphic axiomatizations. For example, it may be specified by its \emph{bases}, which satisfy the basis-exchange axiom, or by its \emph{circuits}, which satisfy the circuit elimination axiom. Baker and Bowler~\cite{BB} extended matroid theory to $F$-matroids over a \emph{tract} $F$. When $F$ is a field, the regular partial field, the Krasner hyperfield, the tropical hyperfield, or the sign hyperfield, $F$-matroids specialize to realizable matroids, regular matroids, matroids, valuated matroids, or oriented matroids, respectively. Baker and Bowler established several cryptomorphisms for $F$-matroids, including the equivalence between their basis and circuit axiomatizations. A key ingredient in their proof is Maurer's homotopy theorem~\cite{Maurer1973} for the \emph{basis graph} of a matroid. 

\begin{thm}[Maurer~\cite{Maurer1973}]\label{thm: Maurer homotopy theorem}
    In the basis graph of a matroid, any two walks with the same end-points are combinatorially homotopic.
\end{thm}

Analogous Baker--Bowler theories have been developed for \emph{orthogonal matroids} and \emph{antisymmetric matroids}. Orthogonal matroids are equivalent to \emph{even delta-matroids}, which generalize matroids and may be viewed as Coxeter matroids of type~D. Jin and the second author~\cite{JK} established cryptomorphisms for orthogonal $F$-matroids using Wenzel's homotopy theorem~\cite{Wenzel1996}, which extends Maurer's theorem from matroids to even delta-matroids. Antisymmetric matroids also generalize matroids and are closely related to Coxeter matroids of type~C. In~\cite{Kim}, the second author introduced antisymmetric matroids, proved a homotopy theorem for them,\footnote{In \cite{Wenzel1996} and \cite{Kim}, the concepts of homology and homotopy are conflated. See Section~\ref{sec: homology} for a detailed discussion.} and applied it to establish the corresponding cryptomorphisms. This homotopy theorem, in turn, generalizes Wenzel's theorem.

Our main contribution is a proof of Theorem~\ref{thm: Maurer homotopy theorem} based on the geometric characterization of matroids provided by their \emph{basis polytopes}. This proof is more geometric and conceptual than the original combinatorial argument. The same approach yields a homotopy theorem for delta-matroids that encompasses all the results discussed above.
\begin{thm}\label{thm: homotopy theorem for delta matroids}
    In the basis graph of a delta-matroid (see Definition~\ref{def: basis graph of delta-matroid}), any two walks with the same end-points are combinatorially homotopic.
\end{thm}

Another natural generalization is given by \emph{integral polymatroids}, whose bases are integer vectors rather than sets. The notion of an $F$-polymatroid was introduced in~\cite{BHKKL2025} in terms of extended Pl\"ucker vectors and was subsequently used in~\cite{BHKL2025b} to study the space of Lorentzian polynomials. We also establish a homotopy theorem for integral polymatroids, which is a key ingredient in the proof of a cryptomorphism for $F$-polymatroids in~\cite{KP2026}.
\begin{thm}\label{thm: polymatroid homotopy theorem}
In the basis graph of an integral polymatroid, any two walks with the same end-points are combinatorially homotopic.
\end{thm}

This paper is self-contained. To provide an overview of the argument, we sketch the proof in the matroid case.
\begin{enumerate}[label=(\arabic*)]
    \item For a connected simple graph $G$, any two walks with the same end-points are combinatorially homotopic if and only if $\pi_1(X^{34}_G)$ is trivial, where $X^{34}_G$ is the topological space obtained from $G$ by attaching $2$-cells along the boundaries of its $3$- and $4$-cycles.
    \item The basis graph $G_M$ of a matroid $M$ is the $1$-skeleton of the basis polytope of $M$.
    \item The $2$-skeleton of the basis polytope is simply connected.
    \item Every $2$-face of the basis polytope is either a triangle or a square.
    \item Together, the preceding facts imply Theorem~\ref{thm: Maurer homotopy theorem}.
    
\end{enumerate}

The paper is organized as follows. In Section~\ref{sec: Maurer homotopy}, we recall Maurer's definition of combinatorial homotopy and develop a polyhedral tool for proving combinatorial homotopy theorems.
In Sections~\ref{sec: matroid}, \ref{sec: delta-matroid}, and \ref{sec: polymatroid}, we apply this tool to matroids (Theorem~\ref{thm: Maurer homotopy theorem}), delta-matroids (Theorem~\ref{thm: homotopy theorem for delta matroids}), and polymatroids (Theorem~\ref{thm: polymatroid homotopy theorem}), respectively. In Section~\ref{antisymmetric matroid}, we deduce the homotopy theorem for antisymmetric matroids (Theorem~\ref{thm: Maurer homotopy theorem for antisymmetric matroids}) from the delta-matroid result.

Our proofs rely heavily on the Gelfand--Serganova theorem for delta-matroids. Because a proof is difficult to extract from the existing literature, we include one in Section~\ref{sec: GGMS for delta-matroids}.

Historically, Wenzel recast Maurer's homotopy theorem as a homological statement, referring to it as an ``algebraic reformulation'' of Maurer's theorem; see the paragraph preceding~\cite[Theorem~1.12]{Wenzel1996}. It is unclear to us, however, why the homological statement should imply the homotopical statement. We discuss this issue in detail in Section~\ref{sec: homology}. Finally, in Section~\ref{sec: high dim}, we briefly discuss a higher-dimensional generalization of the homological property.

\section{A homotopy theory for graphs}\label{sec: Maurer homotopy}

We assume that all graphs are undirected, simple, and connected. The graphs of interest to this paper are the basis graphs of matroids and delta-matroids, and they satisfy the assumption. 

A \emph{walk} in a graph is a sequence of vertices such that any two consecutive vertices are adjacent. A walk is \emph{closed} if its first and last vertices coincide. When we want to highlight that the first (and last) vertex of a closed walk is $v$, the closed walk is said to be \emph{$v$-based}.

We review Maurer's original definition of homotopy for graphs; see also Figure~\ref{fig:elementary-transformations}. 
\begin{defn}[\cite{Maurer1973}]
 Let $G$ be a graph and let $W_1$ and $W_2$ be two walks in $G$. 
 \begin{enumerate}[label=\rm(\roman*)]
    \item If $W_1 = v_0 \cdots v_{k-1} v_{k} v_{k+1} \cdots v_l$ and $W_2 = v_0 \cdots v_{k-1} v_{k+2} \cdots v_l$
for some $k$ with $v_{k-1} = v_{k+1}$, we say that $W_1$ and $W_2$ differ by a \emph{deletion}.
    \item If $W_1 = v_0 \cdots v_{k-1} v_k v_{k+1} \cdots v_l$ and $W_2 = v_0 \cdots v_{k-1} v_{k+1} \cdots v_l$
for some $k$ with $v_{k-1} \ne v_{k+1}$, we say that $W_1$ and $W_2$ differ by a \emph{shortcut}.
    \item If $W_1 = v_0 \cdots v_{k-1} v_k v_{k+1} \cdots v_l$ and $W_2 = v_0 \cdots v_{k-1} v_k' v_{k+1} \cdots v_l$
for some $k$ and distance-$2$ vertices $v_{k-1}$ and $v_{k+1}$, 
        we say that $W_1$ and $W_2$ differ by a \emph{2-switch}.
\end{enumerate}
These three are called \emph{elementary transformations} on walks.
Two walks are \emph{combinatorially homotopic}\footnote{In Maurer's paper \cite{Maurer1973}, ``combinatorially homotopic'' is just ``homotopic''. We add the adjective because later we will also use algebraic topology. } if one can be obtained from the other by a sequence of elementary transformations. A $v$-based closed walk is \emph{combinatorially null-homotopic} if it is combinatorially homotopic to $v$ (viewed as a trivial walk).  

\end{defn}

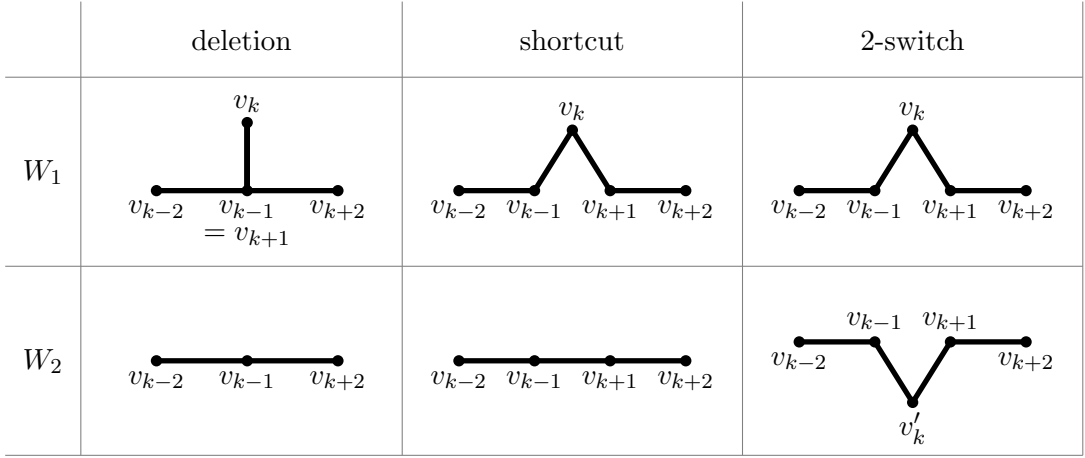
\begin{figure}[!ht]
    \centering
    \begin{tikzpicture}
        \draw[color=gray] 
            (0,0) -- (14.25,0)
            (0,-2.5) -- (14.25,-2.5)
            (0,-5.0) -- (14.25,-5.0);

        \draw[color=gray] 
            (1,1) -- (1,-5)
            (5.25,1) -- (5.25,-5)
            (9.75,1) -- (9.75,-5)
            (14.25,1) -- (14.25,-5);

        \node at (0.5,-1.25) {$W_1$};
        \node at (0.5,-3.75) {$W_2$};

        \node at (3.125,0.5) {deletion};
        \node at (7.5,0.5) {shortcut};
        \node at (12.0,0.5) {$2$-switch};

        \begin{scope}[xshift=2cm, yshift=-1.5cm]
            \coordinate (v0) at (0,0);
            \coordinate (v1) at (1.2,0);
            \coordinate (w) at (1.2,0.9);
            \coordinate (v2) at (2.4,0);

            \node at (v0) [circle,fill,inner sep=1.5pt] {};
            \node at (v1) [circle,fill,inner sep=1.5pt] {};
            \node at (w) [circle,fill,inner sep=1.5pt] {};
            \node at (v2) [circle,fill,inner sep=1.5pt] {};

            \draw[line width=0.7mm] (v0) -- (v1) -- (w) -- (v1) -- (v2);

            \draw
                (v0) node[below] {$v_{k-2}$}
                (v1) node[below] {$v_{k-1}$}
                (v1) node[below, yshift=-3.5mm] {$=v_{k+1}$}
                (w) node[above] {$v_{k}$}
                (v2) node[below] {$v_{k+2}$};
        \end{scope}
        \begin{scope}[xshift=2cm,yshift=-3.75cm]
            \coordinate (v0) at (0,0);
            \coordinate (v1) at (1.2,0);
            \coordinate (w) at (1.2,1.0);
            \coordinate (v2) at (2.4,0);

            \node at (v0) [circle,fill,inner sep=1.5pt] {};
            \node at (v1) [circle,fill,inner sep=1.5pt] {};
\node at (v2) [circle,fill,inner sep=1.5pt] {};

            \draw[line width=0.7mm] (v0) -- (v1) -- (v2);

            \draw
                (v0) node[below] {$v_{k-2}$}
                (v1) node[below] {$v_{k-1}$}
(v2) node[below] {$v_{k+2}$};
        \end{scope}

        \begin{scope}[xshift=6cm,yshift=-1.5cm]
            \coordinate (v0) at (0,0);
            \coordinate (v1) at (1.0,0);
            \coordinate (w) at (1.5,0.8);
            \coordinate (v2) at (2.0,0);
            \coordinate (v3) at (3.0,0);

            \node at (v0) [circle,fill,inner sep=1.5pt] {};
            \node at (v1) [circle,fill,inner sep=1.5pt] {};
            \node at (w) [circle,fill,inner sep=1.5pt] {};
            \node at (v2) [circle,fill,inner sep=1.5pt] {};
            \node at (v3) [circle,fill,inner sep=1.5pt] {};

            \draw[line width=0.7mm] (v0) -- (v1) -- (w) -- (v2) -- (v3);

            \draw
                (v0) node[below] {$v_{k-2}$}
                (v1) node[below] {$v_{k-1}$}
                (w) node[above] {$v_{k}$}
                (v2) node[below] {$v_{k+1}$}
                (v3) node[below] {$v_{k+2}$};
        \end{scope}
        \begin{scope}[xshift=6cm, yshift=-3.75cm]
            \coordinate (v0) at (0,0);
            \coordinate (v1) at (1.0,0);
            \coordinate (w) at (1.5,0.8);
            \coordinate (v2) at (2.0,0);
            \coordinate (v3) at (3.0,0);

            \node at (v0) [circle,fill,inner sep=1.5pt] {};
            \node at (v1) [circle,fill,inner sep=1.5pt] {};
\node at (v2) [circle,fill,inner sep=1.5pt] {};
            \node at (v3) [circle,fill,inner sep=1.5pt] {};

            \draw[line width=0.7mm] (v0) -- (v1) -- (v2) -- (v3);

            \draw
                (v0) node[below] {$v_{k-2}$}
                (v1) node[below] {$v_{k-1}$}
(v2) node[below] {$v_{k+1}$}
                (v3) node[below] {$v_{k+2}$};
        \end{scope}

        \begin{scope}[xshift=10.5cm,yshift=-1.5cm]
            \coordinate (v0) at (0,0);
            \coordinate (v1) at (1.0,0);
            \coordinate (w) at (1.5,0.8);
            \coordinate (v2) at (2.0,0);
            \coordinate (v3) at (3.0,0);

            \node at (v0) [circle,fill,inner sep=1.5pt] {};
            \node at (v1) [circle,fill,inner sep=1.5pt] {};
            \node at (w) [circle,fill,inner sep=1.5pt] {};
            \node at (v2) [circle,fill,inner sep=1.5pt] {};
            \node at (v3) [circle,fill,inner sep=1.5pt] {};

            \draw[line width=0.7mm] (v0) -- (v1) -- (w) -- (v2) -- (v3);

            \draw
                (v0) node[below] {$v_{k-2}$}
                (v1) node[below] {$v_{k-1}$}
                (w) node[above] {$v_{k}$}
                (v2) node[below] {$v_{k+1}$}
                (v3) node[below] {$v_{k+2}$};
        \end{scope}
        \begin{scope}[xshift=10.5cm,yshift=-3.5cm]
            \coordinate (v0) at (0,0);
            \coordinate (v1) at (1.0,0);
            \coordinate (w) at (1.5,-0.8);
            \coordinate (v2) at (2.0,0);
            \coordinate (v3) at (3.0,0);

            \node at (v0) [circle,fill,inner sep=1.5pt] {};
            \node at (v1) [circle,fill,inner sep=1.5pt] {};
            \node at (w) [circle,fill,inner sep=1.5pt] {};
            \node at (v2) [circle,fill,inner sep=1.5pt] {};
            \node at (v3) [circle,fill,inner sep=1.5pt] {};

            \draw[line width=0.7mm] (v0) -- (v1) -- (w) -- (v2) -- (v3);

            \draw
                (v0) node[below] {$v_{k-2}$}
                (v1) node[above] {$v_{k-1}$}
                (w) node[below] {$v_{k}'$}
                (v2) node[above] {$v_{k+1}$}
                (v3) node[below] {$v_{k+2}$};
        \end{scope}

    \end{tikzpicture}
    \caption{Local pictures of the three elementary transformations on walks.}
    \label{fig:elementary-transformations}
\end{figure}

\begin{rmk}\label{rmk: fake 2-switch}
In a $2$-switch, the vertices $v_{k-1}$ and $v_{k+1}$ having distance-$2$ implies that there is no edge connecting them. However, when there is such an edge, $W_1$ and $W_2$ in (iii) are still combinatorially homotopic by two shortcuts. 
\end{rmk}

We relate Maurer's notion of homotopy to the fundamental group of a graph, defined combinatorially as in \cite[Section~2.5]{TGT}. Let $W_1$ and $W_2$ be two walks in a graph $G$ such that the first vertex of $W_2$ equals the last vertex of $W_1$. Then the product walk $W_1W_2$ is constructed by extending the vertex sequence of $W_1$ by the vertex sequence of $W_2$ (and merging the two middle equal vertices). Fix a base vertex $v$ of $G$.
Certainly, the product of two $v$-based closed walks is still $v$-based. Two $v$-based closed walks $W_1$ and $W_2$ are called \emph{equivalent} if there is a sequence $W_1,W^{(1)},\cdots, W^{(k)},W_2$ of $v$-based closed walks such that any two consecutive walks differ by a deletion (in Maurer's sense). We denote by $[W]$ the equivalence class of a $v$-based closed walk $W$. If a walk does not contain three consecutive vertices of the form $v_1v_2v_1$, then the walk is said to be \emph{reduced}. It is easy to check that every $v$-based closed walk is equivalent to a unique reduced walk. The equivalence classes of $v$-based closed walks form a group where $[W_1]\cdot[W_2]:=[W_1W_2]$. We also have $[W]^{-1}=[W^{-1}]$, where $W^{-1}$ is the reverse walk of $W$. We call this group \emph{the fundamental group of the graph $G$ based at $v$}, denoted by $\pi_1(G,v)$. 

We define a \emph{$k$-cycle} as a closed walk whose vertex sequence has exactly $k+1$ vertices and exactly $k$ distinct vertices. 
For the fixed base vertex $v$ of $G$, let \[N_v^{34}\lhd \pi_1(G,v)\] be the subgroup of $\pi_1(G,v)$ generated by the elements of the form $[WCW^{-1}]$, where $C$ is a $3$-cycle or a $4$-cycle and $W$ is a walk from $v$ to the base vertex of $C$ (which might not be the same as the base vertex of $G$). It is easy to check by definition that $N_v^{34}$ is a normal subgroup. 

After showing the following two lemmas, we will see that the combinatorial homotopy can be characterized in terms of the group $N_v^{34}$.

\begin{lem}\label{lem: 34cyclesaretrivial}
Let $C$ be a $3$-cycle or a $4$-cycle and $W$ be a walk from a vertex $v$ to the base vertex $v_1$ of $C$. Then the $v$-based closed walk $WCW^{-1}$ is combinatorially null-homotopic. 
\end{lem}
\begin{proof}
It is enough to show that $C$ is combinatorially homotopic to the trivial walk $v_1$. 

When $C$ is a $3$-cycle, denote $C=v_1v_2v_3v_1$. Then $C$ is combinatorially homotopic to the closed walk $v_1v_3v_1$ by a shortcut and then to $v_1$ by a deletion. 

When $C$ is a $4$-cycle, denote $C=v_1v_2v_3v_4v_1$. If $v_1v_3$ is an edge, then $C$ can be decomposed into two $3$-cycles, and we may apply the argument for $3$-cycles twice to prove the desired result. Otherwise, $C$ is combinatorially homotopic to the closed walk $v_1v_4v_3v_4v_1$ by a $2$-switch and then to $v_1$ by two deletions. 
\end{proof}
\begin{lem}
Let $W_1$ and $W_2$ be two $v$-based closed walks. Then $W_1$ is combinatorially homotopic to $W_2$ if and only if $[W_1]\cdot[W_2]^{-1}\in N_v^{34}$ (or equivalently, $[W_1W_2^{-1}]\in N_v^{34}$). 
\end{lem}
\begin{proof}
By definition, $W_1$ and $W_2$ differ by deletions if and only if $[W_1]=[W_2]$. If $W_1$ and $W_2$ differ by a shortcut or a $2$-switch, then it is easy to see that $W_1W_2^{-1}$ is equivalent to a walk of the form $WCW^{-1}$, where $C$ is a $3$-cycle or a $4$-cycle. Hence $[W_1]\cdot[W_2]^{-1}=[W_1W_2^{-1}]=[WCW^{-1}]$. Thus the ``only if'' part holds. 

For the ``if'' part, it is enough to show that if $[W_1]=[WCW^{-1}]\cdot[W_2]$ for a generator $[WCW^{-1}]$ of the group $N_v^{34}$, then $W_1$ is combinatorially homotopic to $W_2$. We know that this holds by Lemma~\ref{lem: 34cyclesaretrivial}. 
\end{proof}

Consequently, we have the following result. \begin{cor}\label{cor: group Maurer}
Let $G$ be a graph and let $v$ be a vertex. Then every $v$-based closed walk is combinatorially null-homotopic if and only if $\pi_1(G,v)/N_v^{34}$ is a trivial group.
\end{cor}

We may view a graph $G$ as a CW complex, denoted by $X_G$, where the $0$-cells and the $1$-cells are the vertices and edges, respectively. 

\begin{defn}\label{def: X_G}
For a graph $G$, let $X^{34}_G$ be the topological space obtained from $G$ by attaching $2$-cells along the boundary of each $3$- and $4$-cycle of $G$. (In this definition, cycles with the same underlying edge set despite having different base vertices are considered to be the same.)
\end{defn}

\begin{lem}\label{lem: two pie1}
We have a group isomorphism $\pi_1(G,v)/N_v^{34}\cong\pi_1(X^{34}_G,v)$.\end{lem}
\begin{proof}
The fundamental group $\pi_1(G,v)$ is isomorphic to the usual fundamental group where $G$ is viewed as a CW complex (\cite[Thm~3.6.17]{Spanier}). Then \cite[Prop.~1.26(a)]{Hatcher} implies the isomorphism. \end{proof}

Since the fundamental group of a path-connected topological space is, up to isomorphism, independent of the choice of basepoint, so are $\pi_1(G,v)$ and $\pi_1(G,v)/N_v^{34}$.

\begin{prop}\label{prop: topology Maurer}
Let $G$ be a connected simple graph. Then the following are equivalent:
\begin{enumerate}[label=\rm(\roman*)]
    \item any two walks with the same end-points are combinatorially homotopic;
    \item for any base vertex $v$, any $v$-based closed walk is combinatorially null-homotopic;
    \item there exists a vertex $v$ such that any $v$-based closed walk is combinatorially null-homotopic;
    \item $\pi_1(X^{34}_G)$ is a trivial group. 
\end{enumerate}
\end{prop}

\begin{proof}
    The equivalences between (ii), (iii), and (iv) are implied by Corollary~\ref{cor: group Maurer} and Lemma~\ref{lem: two pie1}. Note that (ii) is a special case of (i). It remains to show that (ii) implies (i). Let $W_1$ and $W_2$ be two walks from $u$ to $w$. Then $W_1$ is combinatorially homotopic to $W_1(W_2^{-1}W_2)$ by definition, and  $(W_1W_2^{-1})W_2$ is combinatorially homotopic to $W_2$ by (ii).  
\end{proof}

So far none of the results in this section is essentially new; see the following two remarks. 

\begin{rmk}
Lemma~\ref{lem: two pie1} also appears in A-homotopy theory but in a slightly different form. For example, \cite[Prop~5.12]{BKL2001} says that the first \emph{graph homotopy group} of $G$ is isomorphic to $\pi_1(X^{34}_G)$. We do not use their notion to build our theory because it is technical to define the graph homotopy group. Also, although it is well-known in A-homotopy theory that two walks are combinatorially homotopic (in Maurer's sense) if and only if they are A-homotopic, we only find a proof for the ``only if'' direction in \cite[Section 5.1]{BL2005}, and the proof for the other direction does not seem very trivial to us. Based on these considerations, we use the current definition of $\pi_1(G,v)$ rather than their graph homotopy group to relate Maurer's homotopy to $\pi_1(X^{34}_G)$. We believe that this approach is more direct and simple. 
\end{rmk}

\begin{rmk}
The space $X^{34}_G$ also appeared in \cite{Lovasz} and \cite{CCO2015}. 
\end{rmk}

We will use a geometric approach (Proposition~\ref{prop: main}) to prove $\pi_1(X^{34}_G)=1$ when $G$ is the basis graph of a matroid or a delta-matroid. Before that, we recall some basics of polytope theory. For details, we refer the reader to Ziegler's book \cite{Ziegler1995}.

A (convex and bounded) \emph{polytope} $\pP$ is the convex hull of finitely many points in a Euclidean space. A \emph{face} of a polytope is the set of all points within the polytope that maximize the value of some linear functional.
The faces of dimension $i$ are called $i$-faces. The $0$-faces and $1$-faces are called \emph{vertices} and \emph{edges}, respectively. The vertices and edges of a polytope $\pP$ form an (undirected) graph $G$, called the \emph{graph of the polytope} $\pP$. It is well known that the graph $G$ is connected. 

\begin{prop}\label{prop: main}
Let $\pP$ be a polytope, and denote its graph by $G$. If every $2$-face of $\pP$ is either a triangle or a quadrilateral, then any two walks with the same end-points in $G$ are combinatorially homotopic.
\end{prop}
\begin{proof}
By Proposition~\ref{prop: topology Maurer}, we need to show that the group $\pi_1(X^{34}_G)$ is trivial. 

Consider the $2$-skeleton $\pP^2$ of the polytope $\pP$. Because $\pP$ can be obtained from $\pP^2$ by attaching $n$-cells for $n>2$, by \cite[Prop.~1.26(b)]{Hatcher}, we have an isomorphism \[\pi_1(\pP^2)\cong\pi_1(\pP).\]
Since $\pi_1(\pP)$ is clearly trivial, so is $\pi_1(\pP^2)$. 

Because every $2$-face of $\pP$ is either a triangle or a square, the CW complex $X^{34}_G$ can be obtained from $\pP^2$ by attaching additional $2$-cells. By \cite[Prop.~1.26(a)]{Hatcher}, we have a surjection \[\pi_1(\pP^2)\twoheadrightarrow\pi_1(X^{34}_G),\]  
which means attaching $2$-cells imposes more relations on the fundamental group. Thus $\pi_1(X^{34}_G)$ is trivial. 
\end{proof}

Later when we study polymatroids, we will need a generalization of Proposition~\ref{prop: main} from polytopes to polyhedral complexes. 

\begin{defn}\label{def: polyhedral complex}
A \emph{polyhedral complex} $\pC$ is a non-empty finite set of polytopes in a Euclidean space satisfying the following conditions:
\begin{enumerate}
    \item Every face of a polytope in $\pC$ is in $\pC$. 
    \item For any $\pP_1,\pP_2\in \pC$, $\pP_1\cap\pP_2$ is empty or a common face of both $\pP_1$ and $\pP_2$. 
\end{enumerate}
An $i$-dimensional polytope in $\pC$ is called an $i$-face of $\pC$. The $0$-faces and $1$-faces of $\pC$ form a graph, called the \emph{graph of the polyhedral complex} $\pC$.
\end{defn}

\begin{prop}\label{prop: main2}
Let $\pC$ be a polyhedral complex such that the union of the faces in $\pC$ is a simply connected topological space. Denote its graph by $G$. Suppose every $2$-face of $\pC$ is either a triangle or a quadrilateral. Then any two walks with the same end-points in $G$ are combinatorially homotopic.
\end{prop}
The proof is the same as that of Proposition~\ref{prop: main}, except that the polytope $\pP$ is replaced by the polyhedral complex $\pC$.

\section{Maurer's homotopy theorem for matroids}\label{sec: matroid}
Let $[n]=\{1,2,\cdots,n\}$.
For a non-empty collection $\cB$ of subsets of $[n]$, we call the pair $([n],\cB)$ a non-empty \emph{set system}. For $B\in \cB$, $e\in [n]\setminus B$, and $f\in B$, let\[B+e:=B\cup\{e\} \quad\text{and}\quad B-f:=B\setminus\{f\}.\]

For the purpose of this paper, we define a matroid using the basis exchange property.

\begin{defn}A \emph{matroid} $M$ on $[n]$ is a non-empty set system $([n], \cB)$ such that for all $B,B'\in \cB$ and $e\in B\setminus B'$, there is an element $f\in B'\setminus B$ such that $B-e+f \in \cB$. An element in $\cB$ is called a \emph{basis}. 
\end{defn}

We often identify a collection $\cB$ of subsets of $[n]$ with the set of characteristic vectors $\bfe_B := \sum_{i\in B} \bfe_i$ with $B\in \cB$. The convex hull of these vectors is denoted by\[ \pP_\cB := \mathrm{conv}\{\bfe_B\colon B\in \cB\}.\]

\begin{defn}
    The \emph{basis polytope} $\pP_M$ of a matroid $M=([n], \cB)$ is the convex hull of the characteristic vectors of its bases, i.e., \[\pP_M :=\pP_\cB.\]
\end{defn}

Gelfand, Goresky, MacPherson, and Serganova~\cite{GGMS1987} characterized the basis polytopes of matroids by their edges. Recall that a \emph{0/1-polytope} is the convex hull of points with coordinate values 0 or 1. Note that these points are the vertices of the 0/1-polytope.

\begin{thm}[\cite{GGMS1987}]\label{thm:GGMS}
    Let $\pP$ be a 0/1-polytope. Then $\pP$ is the basis polytope of a matroid if and only if every edge is a translate of a vector of the form $\bfe_i-\bfe_j$ for some $i\ne j$. 
\end{thm}

We will also need an analogous theorem for delta-matroids, which is in another paper of Gelfand and Serganova. Since it is inconvenient for the readers to find the proofs in different papers, we give a proof for these results in Section~\ref{sec: GGMS for delta-matroids}.

The following lemma on the $2$-faces of a basis polytope is well known and follows immediately from Theorem~\ref{thm:GGMS}. 
We include the proof for the sake of completeness.

\begin{lem}\label{lem: 2-face} A $2$-face of the basis polytope of a matroid is either a triangle or a square; see Figure~\ref{fig:2-face of matroid polytope}.
\end{lem}
\begin{proof}
Let $F$ be a $2$-face, which is a polygon. Pick any edge of $F$ and denote its end-points by $A$ and $B$. Then by Theorem~\ref{thm:GGMS}, the vector $\overrightarrow{AB}$ equals $\bfe_i-\bfe_j$ for some $i\ne j$. Denote by $BC$ the other edge of the polygon incident to $B$. Then the vector $\overrightarrow{BC}$ can only be $\bfe_k-\bfe_l$, $\bfe_j-\bfe_k$, or $\bfe_k-\bfe_i$ for some $k,l\notin\{i,j\}$ because the matroid polytope is a 0/1-polytope. Since all the other edges are also of the form $\bfe_\star-\bfe_\bullet$ and are in the space generated by $\overrightarrow{AB}$ and $\overrightarrow{BC}$, it is easy to see that the face $F$ must be one of the three cases in Figure~\ref{fig:2-face of matroid polytope}. A more precise argument is as follows.  
\begin{enumerate}
    \item When $\overrightarrow{BC}=\bfe_k-\bfe_l$, the next vector, denoted by $\overrightarrow{CD}$ (possibly $D=A$), is a linear combination of $\bfe_i-\bfe_j$ and $\bfe_k-\bfe_l$. Because $\overrightarrow{CD}$ is of the form $\bfe_\star-\bfe_\bullet$, we have $\overrightarrow{CD}=\pm(\bfe_i-\bfe_j)$ or $\pm(\bfe_k-\bfe_l)$. Since $AB,BC,CD$ are three consecutive edges of a $2$-dimensional convex polytope, the only choice is $\overrightarrow{CD}=\bfe_j-\bfe_i$. The vector next to $\overrightarrow{CD}$, by the same argument, must be $\pm(\bfe_i-\bfe_j)$ or $\pm(\bfe_k-\bfe_l)$. The only choice is $\bfe_l-\bfe_k$, and hence the face $F$ is the square $ABCD$. 
    
    \item When $\overrightarrow{BC}=\bfe_j-\bfe_k$, the next vector, denoted by $\overrightarrow{CD}$, is a linear combination of $\bfe_i-\bfe_j$ and $\bfe_j-\bfe_k$. Because $\overrightarrow{CD}$ is of the form $\bfe_\star-\bfe_\bullet$, we have $\overrightarrow{CD}=\pm(\bfe_i-\bfe_j)$, $\pm(\bfe_j-\bfe_k)$, or $\pm(\bfe_i-\bfe_k)$. Since $AB,BC,CD$ are three consecutive edges of a $2$-dimensional convex polytope, the choices are $\overrightarrow{CD}=\bfe_j-\bfe_i$ or $\bfe_k-\bfe_i$. Since it is a 0/1-polytope, $\overrightarrow{CD}\neq\bfe_j-\bfe_i$. Therefore, $D=A$ and the face $F$ is the triangle $ABC$.
    \item The case $\overrightarrow{AC}=\bfe_k-\bfe_i$ is similar to the previous one. \qedhere
\end{enumerate}
\end{proof}

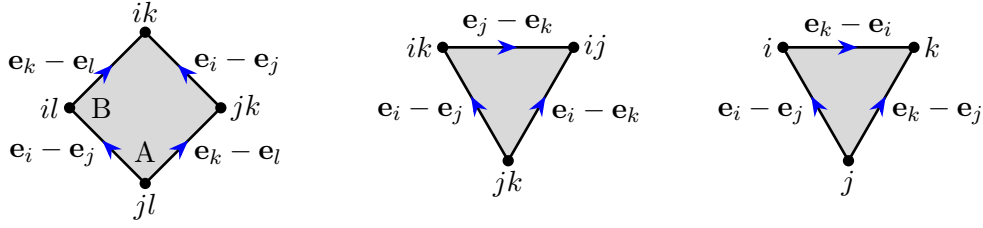
\begin{figure}
    \centering
    \begin{tikzpicture}
       \begin{scope}
    \coordinate (a) at (-90:1.0);
    \coordinate (b) at (180:1.0);
    \coordinate (c) at (90:1.0);
    \coordinate (d) at (0:1.0);

    \fill[gray!30] (a) -- (b) -- (c) -- (d) -- cycle;

    \node at (a) [circle,fill,inner sep=1.5pt] {};
    \node at (b) [circle,fill,inner sep=1.5pt] {};
    \node at (c) [circle,fill,inner sep=1.5pt] {};
    \node at (d) [circle,fill,inner sep=1.5pt] {};

\node[above=0.15cm] at (a) {A};
\node[right=0.15cm] at (b) {B};

    \draw
        (a) node[below] {$jl$}
        (b) node[left] {$il$}
        (c) node[above] {$ik$}
        (d) node[right] {$jk$};

    \draw[midarrow, line width=1pt] (a) --node[left,yshift=-1mm]{$\bfe_i- \bfe_j$} (b);
    \draw[midarrow, line width=1pt] (b) --node[left,yshift=1mm]{$\bfe_k- \bfe_l$} (c);
    \draw[midarrow, line width=1pt] (a) --node[right,yshift=-1mm]{$\bfe_k- \bfe_l$} (d);
    \draw[midarrow, line width=1pt] (d) --node[right,yshift=1mm]{$\bfe_i- \bfe_j$} (c);
\end{scope}

        \begin{scope}[xshift=4.8cm, yshift=0.3cm]
            \coordinate (a) at (-90:1.0);
            \coordinate (b) at (150:1.0);
            \coordinate (c) at (30:1.0);

            \fill[gray!30] (a) -- (b) -- (c) -- cycle;

            \node at (a) [circle,fill,inner sep=1.5pt] {};
            \node at (b) [circle,fill,inner sep=1.5pt] {};
            \node at (c) [circle,fill,inner sep=1.5pt] {};

            \draw
                (a) node[below] {$jk$}
                (b) node[left] {$ik$}
                (c) node[right] {$ij$};

            \draw[midarrow, line width=1pt] (a) --node[left,yshift=-1mm]{$\bfe_i- \bfe_j$} (b);
            \draw[midarrow, line width=1pt] (b) --node[above]{$\bfe_j- \bfe_k$} (c);
            \draw[midarrow, line width=1pt] (a) --node[right,yshift=-1mm]{$\bfe_i- \bfe_k$} (c);
        \end{scope}

        \begin{scope}[xshift=9.3cm, yshift=0.3cm]
            \coordinate (a) at (-90:1.0);
            \coordinate (b) at (150:1.0);
            \coordinate (c) at (30:1.0);

            \fill[gray!30] (a) -- (b) -- (c) -- cycle;

            \node at (a) [circle,fill,inner sep=1.5pt] {};
            \node at (b) [circle,fill,inner sep=1.5pt] {};
            \node at (c) [circle,fill,inner sep=1.5pt] {};

            \draw
                (a) node[below] {$j$}
                (b) node[left] {$i$}
                (c) node[right] {$k$};

            \draw[midarrow, line width=1pt] (a) --node[left,yshift=-1mm]{$\bfe_i- \bfe_j$} (b);
            \draw[midarrow, line width=1pt] (b) --node[above]{$\bfe_k- \bfe_i$} (c);
            \draw[midarrow, line width=1pt] (a) --node[right,yshift=-1mm]{$\bfe_k- \bfe_j$} (c);
        \end{scope}
    \end{tikzpicture}
    \caption{Possible 2-faces of the basis polytope of a matroid.}
    \label{fig:2-face of matroid polytope}
\end{figure}

Now we relate the basis polytope to the basis graph. 
\begin{defn}
The \emph{basis graph} $G_M$ of a matroid $M=([n], \cB)$ is a graph on $\cB$ such that two vertices $B$ and $B'$ are
adjacent if and only if $|B\setminus B'|=1$. 
\end{defn}

The following lemma is also well known. 
\begin{lem}\label{lem: 1-skeleton}
    For a matroid $M$, the graph of its basis polytope $\pP_M$ is its basis graph $G_M$. 
\end{lem}
\begin{proof}
    By definition their vertex sets coincide. 
    Let $AB$ be an edge of $P_M$. Then by Theorem~\ref{thm:GGMS}, $AB$ is a translate of $\bfe_i-\bfe_j$ for some $i\ne j$, which implies $B=A-i+j$ or $A=B-i+j$. Hence $A$ and $B$ are adjacent in $G_M$. Conversely, let $AB$ be an edge of $G_M$. Then by definition, $B=A-i+j$ for some $i\ne j$. Let $f$ be a linear functional such that $f(\bfe_i) = f(\bfe_j) =0$, $f(\bfe_k)=1$ for any $k \in A\cap B$, and $f(\bfe_l)=-1$ for any $l\in [n]\setminus (A\cup B)$. Then $f$ is maximized exactly at the line segment $AB$, which implies that $AB$ is an edge of $\pP_M$.
\end{proof}

By Lemma~\ref{lem: 1-skeleton}, Lemma~\ref{lem: 2-face}, and Proposition~\ref{prop: main}, we immediately obtain Maurer's homotopy theorem:

\begin{thm}[Theorem~\ref{thm: Maurer homotopy theorem}]
Any two walks with the same end-points in the basis graph $G_M$ of a matroid $M$ are combinatorially homotopic.  
\end{thm}

We remark that we only used the ``only if '' direction of Theorem~\ref{thm:GGMS}.

\section{Homotopy theorem for delta-matroids}\label{sec: delta-matroid}
Some notation has been introduced in Section~\ref{sec: matroid}. 
For sets $A$ and $B$, we denote by $A\triangle B$ the symmetric difference of $A$ and $B$, i.e., $A\triangle B=(A\setminus B)\cup (B\setminus A)$.

\begin{defn}
A \emph{delta-matroid} $M$ on $[n]$ is a non-empty set system $([n], \cB)$ such that for all $B,B'\in \cB$ and $e\in B \triangle B'$, there is an element $f\in B\triangle B'$ (possibly, $e=f$) such that $B\triangle\{e,f\} \in \cB$.
A delta-matroid is \emph{even} if all bases have the same parity.
\end{defn}

Evidently, matroids are exactly those delta-matroids whose bases have the same cardinality.

The basis polytope of a delta-matroid generalizes the basis polytope of a matroid. 

\begin{defn}
    The \emph{basis polytope} $\pP_M$ of a delta-matroid $M=([n], \cB)$ is the convex hull of the characteristic vectors of its bases, i.e., \[\pP_M :=\pP_\cB.\]
\end{defn}

The Gelfand--Serganova theorem~\cite{GS1987,BGW1997} for (even) delta-matroids characterizes the basis polytopes of delta-matroids and even delta-matroids.

\begin{thm}[\cite{GS1987,BGW1997}]\label{thm:GGMS for delta-matroids}
    Let $\pP$ be a 0/1-polytope. 
    \begin{enumerate}[label=\rm(\roman*)]
        \item\label{item:GGMS1} The polytope $\pP$ is the basis polytope of a delta-matroid if and only if every edge is a translate of $\bfe_i$ for some $i$, or $\bfe_i-\bfe_j$ or $\bfe_i+\bfe_j$ for some $i\ne j$.
        \item\label{item:GGMS2} The polytope $\pP$ is the basis polytope of an even delta-matroid if and only if every edge is a translate of $\bfe_i-\bfe_j$ or $\bfe_i+\bfe_j$ for some $i\ne j$.
    \end{enumerate}
\end{thm}

Theorem~\ref{thm:GGMS for delta-matroids}\ref{item:GGMS1} follows from Theorems~3.3.3 and 4.1.4 in {\cite{BGW1997}}, and \ref{item:GGMS2} follows from Theorems~3.3.3 and 4.2.4 in {\cite{BGW1997}}. For the sake of completeness, we give a direct proof of Theorem~\ref{thm:GGMS for delta-matroids} in Section~\ref{sec: GGMS for delta-matroids}.

The following lemma lists all types of the $2$-faces of the basis polytopes of delta-matroids.

\begin{lem}\label{lem: 2-face delta}
    A $2$-face of the basis polytope of a delta-matroid is either a triangle or a rectangle; see Figure~\ref{fig:2-face of delta-matroid polytope}.
\end{lem}
\begin{proof}
The proof is just a case analysis, where we apply Theorem~\ref{thm:GGMS for delta-matroids} several times. Let $F$ be a $2$-face, which is a polygon. 
\begin{enumerate}[label=(\roman*)]
    \item Assume that there exists an edge of $F$ parallel to $\bfe_i$. We denote this edge by $AB$ and assume $\overrightarrow{AB}=\bfe_i$. Denote by $BC$ the other edge of the polygon incident to $B$. Then the vector $\overrightarrow{BC}$ can be $\pm \bfe_j$, $-\bfe_i\pm\bfe_j$, $\bfe_j\pm\bfe_k$, or $-\bfe_j\pm\bfe_k$, where $i,j,k$ are distinct. 
    \begin{enumerate}
        \item When $\overrightarrow{BC}=\bfe_j$, the polygon can be closed using one more edge $CA$, which is type~II with $(u_i,u_j)=(1,1)$, or using two more edges, which is type~I. (A more detailed and rigorous proof can be written using the recipe in the proof of Lemma~\ref{lem: 2-face}. We omit such proofs here for simplicity.) 
        \item When $\overrightarrow{BC}=-\bfe_j$, it is similar to the previous case. We either have type~I or type~II with $(u_i,u_j)=(1,-1)$. Note that the two type~I cases are identical, but the two type~II cases are not. 
        \item When $\overrightarrow{BC}=-\bfe_i+\bfe_j$, the polygon can only be closed using $BC$, which is type~II with $(u_i,u_j)=(-1,1)$.
        \item When $\overrightarrow{BC}=-\bfe_i-\bfe_j$, the polygon can only be closed using $BC$, which is type~II with $(u_i,u_j)=(-1,-1)$. By swapping $i$ and $j$, we get type~II with $(u_i,u_j)=(1,1)$.
        \item When $\overrightarrow{BC}=\bfe_j\pm\bfe_k$, the polygon can only be a rectangle, which is type~III. 
        \item When $\overrightarrow{BC}=-\bfe_j\pm\bfe_k$, we still obtain type~III. 
     \end{enumerate}
    \item Assume that there is no edge of $F$ parallel to $\bfe_i$ and there is no edge of $F$ parallel to $\bfe_i+\bfe_j$. Then this is the matroid case which we discussed in Figure~\ref{fig:2-face of matroid polytope}. We obtain type~IV with $(u_j,u_l)=(-1,-1)$ and type~V with $(u_i,u)=\pm(1,-1)$. 
    \item Assume that there is no edge of $F$ parallel to $\bfe_i$ and there exists an edge of $F$ parallel to $\bfe_i+\bfe_j$. We denote this edge by $AB$ and assume $\overrightarrow{AB}=\bfe_i+\bfe_j$. Denote by $BC$ the other edge of the polygon incident to $B$. Then the vector $\overrightarrow{BC}$ can be $\pm(\bfe_k+\bfe_l)$, $\bfe_k-\bfe_l$, or $-\bfe_i\pm\bfe_k$, where $i,j,k,l$ are distinct. \begin{enumerate}
        \item When $\overrightarrow{BC}=\pm(\bfe_k+\bfe_l)$, we obtain type~IV with $(u_j,u_l)=(1,1)$.
        \item When $\overrightarrow{BC}=\bfe_k-\bfe_l$, we obtain type~IV with $(u_j,u_l)=(1,-1)$. 
        \item When $\overrightarrow{BC}=-\bfe_i+\bfe_k$, by swapping $i$ and $j$, we obtain type~V with $(u_i,u)=(1,1)$. 
        \item When $\overrightarrow{BC}=-\bfe_i-\bfe_k$, we obtain type~V with $(u_i,u)=(-1,-1)$. 
        \qedhere
    \end{enumerate}     
\end{enumerate}
\end{proof}

\begin{figure}
    \centering
\begin{tikzpicture}

\begin{scope}
        \coordinate (a) at (0,0);
        \coordinate (b) at (0,1.4);
        \coordinate (c) at (1.4,1.4);
        \coordinate (d) at (1.4,0);

        \fill[gray!30] (a) -- (b) -- (c) -- (d) -- cycle;

        \node at (a) [circle,fill,inner sep=1.5pt] {};
        \node at (b) [circle,fill,inner sep=1.5pt] {};
        \node at (c) [circle,fill,inner sep=1.5pt] {};
        \node at (d) [circle,fill,inner sep=1.5pt] {};

        \draw[midarrow, line width=1pt] (a) --node[left]{$\bfe_i$} (b);
        \draw[midarrow, line width=1pt] (b) --node[above]{$\bfe_j$} (c);
        \draw[midarrow, line width=1pt] (a) --node[below]{$\bfe_j$} (d);
        \draw[midarrow, line width=1pt] (d) --node[right]{$\bfe_i$} (c);

\node[above right] at (a) {A};
\node[below right] at (b) {B};

    \end{scope}

\begin{scope}[xshift=4.9cm]
        \coordinate (a) at (0,0);
        \coordinate (b) at (0,1.4);
        \coordinate (c) at (1.4,1.4);

        \fill[gray!30] (a) -- (b) -- (c) -- cycle;

        \node at (a) [circle,fill,inner sep=1.5pt] {};
        \node at (b) [circle,fill,inner sep=1.5pt] {};
        \node at (c) [circle,fill,inner sep=1.5pt] {};

        \draw[midarrow, line width=1pt] (a) --node[left]{$u_i\bfe_i$} (b);
        \draw[midarrow, line width=1pt] (b) --node[above]{$u_j \bfe_j$} (c);
        \draw[midarrow, line width=1pt] (a) --node[right,yshift=-1mm]{$u_i\bfe_i+u_j\bfe_j$} (c);
    \end{scope}

\begin{scope}[xshift=9.8cm]
        \coordinate (a) at (0,0);
        \coordinate (b) at (0,1.4);
        \coordinate (c) at (2.8,1.4);
        \coordinate (d) at (2.8,0);

        \fill[gray!30] (a) -- (b) -- (c) -- (d) -- cycle;

        \node at (a) [circle,fill,inner sep=1.5pt] {};
        \node at (b) [circle,fill,inner sep=1.5pt] {};
        \node at (c) [circle,fill,inner sep=1.5pt] {};
        \node at (d) [circle,fill,inner sep=1.5pt] {};

        \draw[midarrow, line width=1pt] (a) --node[left]{$\bfe_i$} (b);
        \draw[midarrow, line width=1pt] (b) --node[above]{$\bfe_j +u_k \bfe_k$} (c);
        \draw[midarrow, line width=1pt] (a) --node[below]{$\bfe_j +u_k \bfe_k$} (d);
        \draw[midarrow, line width=1pt] (d) --node[right]{$\bfe_i$} (c);
    \end{scope}

\node at (0.7, -0.8) {type I};
    \node at (5.6, -0.8) {type II};
    \node at (11.2, -0.8) {type III};

\begin{scope}[yshift=-2.8cm]
        
\begin{scope}[xshift=3.15cm]
            \coordinate (a) at (-90:1.0);
            \coordinate (b) at (180:1.0);
            \coordinate (c) at (90:1.0);
            \coordinate (d) at (0:1.0);

            \fill[gray!30] (a) -- (b) -- (c) -- (d) -- cycle;

            \node at (a) [circle,fill,inner sep=1.5pt] {};
            \node at (b) [circle,fill,inner sep=1.5pt] {};
            \node at (c) [circle,fill,inner sep=1.5pt] {};
            \node at (d) [circle,fill,inner sep=1.5pt] {};

            \draw[midarrow, line width=1pt] (a) --node[left,yshift=-1mm]{$\bfe_i+u_j\bfe_j$} (b);
            \draw[midarrow, line width=1pt] (b) --node[left,yshift=1mm]{$\bfe_k+ u_l\bfe_l$} (c);
            \draw[midarrow, line width=1pt] (a) --node[right,yshift=-1mm]{$\bfe_k+u_l\bfe_l$} (d);
            \draw[midarrow, line width=1pt] (d) --node[right,yshift=1mm]{$\bfe_i+u_j\bfe_j$} (c);
        \end{scope}

\begin{scope}[xshift=8.4cm, yshift=0.1cm]
            \coordinate (a) at (-90:1.0);
            \coordinate (b) at (150:1.0);
            \coordinate (c) at (30:1.0);

            \fill[gray!30] (a) -- (b) -- (c) -- cycle;

            \node at (a) [circle,fill,inner sep=1.5pt] {};
            \node at (b) [circle,fill,inner sep=1.5pt] {};
            \node at (c) [circle,fill,inner sep=1.5pt] {};

            \draw[midarrow, line width=1pt] (a) --node[left,yshift=-1mm]{$u_i\bfe_i+u \bfe_j$} (b);
            \draw[midarrow, line width=1pt] (b) --node[above]{$u\bfe_k- u\bfe_j$} (c);
            \draw[midarrow, line width=1pt] (a) --node[right,yshift=-1mm]{$u_i\bfe_i+ u\bfe_k$} (c);
        \end{scope}

\node at (3.15, -1.6) {type IV};
        \node at (8.4, -1.6) {type V};
        
    \end{scope}
\end{tikzpicture}
\caption{Possible 2-faces of the basis polytope of a delta-matroid.}

\vspace{0.2cm}

\begin{minipage}{0.75\textwidth} 
        \small 
        \begin{enumerate}
            \item In type II, the choices $(u_i,u_j)=(1,\pm 1), (-1,1)$ give three subtypes.
            \item In type III, the choices $u_k\in\{\pm 1\}$ give two subtypes.
            \item In type IV, the choices $(u_j,u_l)=(1,\pm 1), (-1,-1)$ give three subtypes.
            \item In type V, the choices $u_i,u\in\{\pm1\}$ give four subtypes.
        \end{enumerate}
    \end{minipage}
\label{fig:2-face of delta-matroid polytope}
\end{figure}
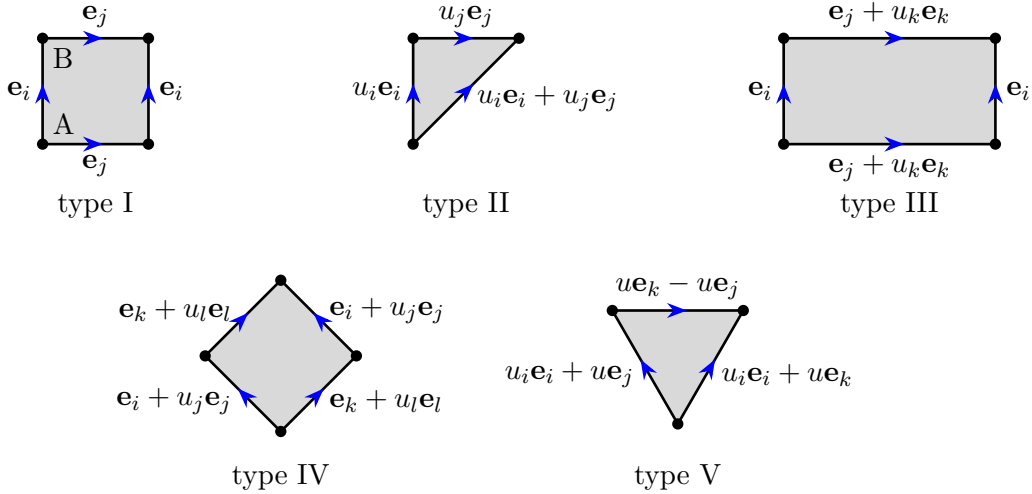

For even delta-matroids, there is no edge parallel to $\bfe_i$. Thus the $2$-faces are types IV and V in Figure~\ref{fig:2-face of delta-matroid polytope}. 

\begin{cor}
A $2$-face of the basis polytope of an even delta-matroid is either a triangle or a square; see types IV and V in Figure~\ref{fig:2-face of delta-matroid polytope}. 
\end{cor}

\begin{rmk}
For each type in Figure~\ref{fig:2-face of delta-matroid polytope} except type I, the subtypes can be considered as the same up to \emph{twisting} of delta-matroid. This is because the twisting at an element $i$ replaces $\bfe_i$ with $-\bfe_i$ in the basis polytope. 
\end{rmk}

Defining the basis graph of a delta-matroid $M$, a notion that is, to the best of our knowledge, new to the literature, requires some care. A natural combinatorial definition would be Definition~\ref{def: basis graph of delta-matroid}. However, the graph in this definition contains certain edges that are not edges of the basis polytope. These are exactly the diagonals of type~I $2$-faces in Figure~\ref{fig:2-face of delta-matroid polytope}. Thus, we will also define the reduced basis graph, which is a combinatorial characterization of deleting these extra edges. 

\begin{defn}\label{def: basis graph of delta-matroid}
    The \emph{basis graph} $G_M$ of a delta-matroid is a graph on $\cB$ such that two vertices $B$ and $B'$ are adjacent if and only if $|B\triangle B'| = 1$ or $2$. 
\end{defn}

\begin{defn}
    The \emph{reduced basis graph} $G'_M$ of a delta-matroid is a graph on $\cB$ such that two vertices $B$ and $B'$ are adjacent if and only if
    \begin{itemize}
        \item $|B\triangle B'| = 1$ or
        \item $|B\triangle B'| = 2$ and there does not exist a pair of distinct vertices $B_1$ and $B_2$ such that
        \begin{equation}\label{eq: reduce}
            |B\triangle B_1|=|B\triangle B_2| =|B'\triangle B_1|=|B'\triangle B_2|= 1.
        \end{equation}
    \end{itemize}
\end{defn}

\begin{lem}\label{lem: 1-skeleton delta}
    For a delta-matroid $M$, the graph of its basis polytope $\pP_M$ is its reduced basis graph $G'_M$. 
\end{lem}
\begin{proof}
    By definition their vertex sets coincide. 
    Let $BB'$ be an edge of $\pP_M$. Then by Theorem~\ref{thm:GGMS for delta-matroids}, $BB'$ is a translate of $\bfe_i$ for some $i$, or $\bfe_i-\bfe_j$ or $\bfe_i+\bfe_j$ for some $i\ne j$. Hence $B\symdiff B'=\{i\}$ or $\{i,j\}$. To conclude that $B$ and $B'$ are adjacent in $G'_M$, we need to show that when $B\symdiff B'=\{i,j\}$, there does not exist a pair of vertices $B_1$ and $B_2$ satisfying \eqref{eq: reduce}. For the sake of contradiction, assume that such $B_1$ and $B_2$ exist. Then the collection of the four bases $B,B',B_1,B_2$ is $\{K,K+i,K+j,K+i+j\}$ for some $K\subseteq [n]-i-j$. Let $f$ be a linear functional such that $f(\bfe_i) = f(\bfe_j) =0$, $f(\bfe_k)=1$ for any $k \in K$, and $f(\bfe_l)=-1$ for any $l\in [n]\setminus(K+i+j)$. Then $f$ is maximized exactly at the square $BB_1B'B_2$. Because $BB'$ is a diagonal of the square $BB_1B'B_2$, the line segment $BB'$ cannot be an edge of $\pP_M$, which gives a contradiction.
    
    Conversely, let $BB'$ be an edge of $G_M'$. There are two cases. \begin{itemize}
        \item When $|B\symdiff B'|=\{i\}$, let $f$ be a linear functional such that $f(\bfe_i)=0$, $f(\bfe_k)=1$ for any $k \in B\cap B'$, and $f(\bfe_l)=-1$ for any $l\in [n]\setminus (B\cup B')$. Then $f$ is maximized exactly at the line segment $BB'$, which implies that $BB'$ is an edge of $\pP_M$. 
        \item When $|B\symdiff B'|=\{i,j\}$, let $f$ be a linear functional such that $f(\bfe_i)=f(\bfe_j)=0$, $f(\bfe_k)=1$ for any $k \in B\cap B'$, and $f(\bfe_l)=-1$ for any $l\in [n]\setminus (B\cup B')$. Then $f$ is maximized at a face $F$ containing the line segment $BB'$. The other vertices of $\pP_M$ that the face $F$ can possibly contain are $B\symdiff\{i\}$ and $B\symdiff\{j\}$, but the condition \eqref{eq: reduce} implies that the face $F$ cannot contain both of them. Hence $F$ is either the line segment $BB'$ or a triangle containing $BB'$ as a side. In either case, $BB'$ is an edge of $\pP_M$. 
        \qedhere
    \end{itemize}
\end{proof}

Now we are ready to prove the analogue of Maurer's homotopy theorem for delta-matroids.
\begin{thm}\label{thm: 2 homotopy theorems for delta-matroids}
Let $M$ be a delta-matroid. 
\begin{enumerate}[label=\rm(\roman*)]
    \item In the reduced basis graph $G'_M$ of $M$, any two walks with the same end-points are combinatorially homotopic.
    \item (Theorem~\ref{thm: homotopy theorem for delta matroids}) In the basis graph $G_M$ of $M$, any two walks with the same end-points are combinatorially homotopic.
\end{enumerate}
\end{thm}
\begin{proof}
The first result follows from Lemma~\ref{lem: 1-skeleton delta}, Lemma~\ref{lem: 2-face delta}, and Proposition~\ref{prop: main}. 
The second result follows from Lemma~\ref{lem: chord} below. 
\end{proof}

\begin{lem}\label{lem: chord}
Let $G'$ be a connected simple graph and $G$ be a graph obtained from $G'$ by adding some chords to the $4$-cycles of $G'$ (i.e., if $v_1v_2v_3v_4v_1$ is a cycle in $G'$, then we may add edges $v_1v_3$ or $v_2v_4$). 
If any two walks with the same end-points are combinatorially homotopic in $G'$, then the same holds in $G$.   
\end{lem}

\begin{proof}
Let $W_1$ and $W_2$ be two walks of $G$ with the same end-points. If both walks are in the subgraph $G'$, then they are combinatorially homotopic, where a $2$-switch might be replaced by two shortcuts (cf. Remark~\ref{rmk: fake 2-switch}). If $W_1$ or $W_2$ contains an edge $e$ not in $G'$, then the walk is combinatorially homotopic to a new walk where $e$ is replaced by edges in $G'$ via a shortcut. Hence each of $W_1$ and $W_2$ is combinatorially homotopic to a walk in $G'$, and the problem is reduced to the previous case.      
\end{proof}

Our homotopy theorem certainly holds for even delta-matroids because they are a subclass of delta-matroids, which recovers Wenzel's result \cite[Theorem~1.12]{Wenzel1996}.

\section{Homotopy theorem for antisymmetric matroids}\label{antisymmetric matroid}

An antisymmetric matroid is a combinatorial structure introduced by the second author~\cite{Kim} to study the combinatorics of the Lagrangian Grassmannian, which has a close relation to a delta-matroid (Lemma~\ref{lem: antisymmetric and delta}).
To define antisymmetric matroids, we introduce some notation. Let \[ [n]^* := \{1^*, 2^*, \ldots, n^*\} \] be a disjoint copy of $[n]$. We let $(i^*)^* = i$ for $i\in [n]$. For a subset $S$ of $[n]\sqcup [n]^*$, let $S^* := \{i^* : i\in S\}$.
We call a $2$-element set $\{i,i^*\}$ a \emph{skew pair}. 
A subset $S$ is a \emph{transversal} if $|S\cap \{i,i^*\}| = 1$ for all $i\in [n]$. An \emph{almost-transversal} is an $n$-element subset $S$ such that $|S\cap \{i,i^*\}| = 1$ for all but two $i\in[n]$, i.e., there are exactly two elements $j,k\in [n]$ such that $|S\cap \{j,j^*\}| = 0$ and $|S\cap \{k,k^*\}| = 2$.
We denote by $\cT_n$ (resp. $\cA_n$) the set of transversals (resp. almost-transversals) of $[n]\sqcup [n]^*$.

\begin{defn}
    An \emph{antisymmetric matroid} $M$ on $[n] \cup [n]^*$ is a non-empty set system $([n]\sqcup [n]^*, \cB)$ such that $\cB$ is a subset of $\cT_n \cup \cA_n$ satisfying the following axiom: for all $B,B'\in \cB$ and $e\in B\setminus B'$, if $B-e$ has no skew pair and $B'+e$ has exactly one skew pair, then there is an element $f\in B'\setminus B$ such that $B-e+f \in \cB$ and $B'+e-f \in \cB$.
    Each element of $\cB$ is called a \emph{basis} of $M$. 
\end{defn}

\begin{lem}[{\cite[Lem.~3.2]{Kim}}]
    \label{lem: 3-term Plucker for antisymmetric}
    Let $M = ([n]\cup[n]^*, \cB)$ be an antisymmetric matroid. For a transversal $T$ and two distinct skew pairs $p$ and $q$, none or at least two of 
    \[
        \{(T\cup p)\setminus q, (T\cup q)\setminus p\},
        \quad
        \{T, T\triangle (p\cup q)\},
        \quad
        \{T\triangle p, T\triangle q\}
    \]
    are subsets of $\cB$. 
\end{lem}

The following relation between antisymmetric matroids and delta-matroids easily follows from Lemma~\ref{lem: 3-term Plucker for antisymmetric}.

\begin{lem}[{\cite[Prop.~4.3]{Kim}}]
    \label{lem: antisymmetric and delta}
    Let $M$ be an antisymmetric matroid on $[n]\sqcup [n]^*$. Let \[ \Delta(\cB) := \{B\cap [n] \colon \text{$B$ is a transversal in $\cB$}\}. \] Then $\Delta(M) := ([n], \Delta(\cB))$ is a delta-matroid.
\end{lem}

We now define a graph associated with an antisymmetric matroid and show a homotopy theorem for antisymmetric matroids (Theorem~\ref{thm: Maurer homotopy theorem for antisymmetric matroids}). The proof will easily follow from the aforementioned relation between antisymmetric matroids and delta-matroids and the homotopy theorem for delta-matroids (Theorem~\ref{thm: 2 homotopy theorems for delta-matroids}). We remark that Theorem~\ref{thm: Maurer homotopy theorem for antisymmetric matroids} implies \cite[Theorem~5.3]{Kim} (see Section~\ref{sec: homology}), which is a key tool to show the equivalence between two cryptomorphic definitions of antisymmetric matroids with coefficients. Compared with our polyhedral approach, the proof of \cite[Theorem~5.3]{Kim} is based on induction on the weight of a cycle and numerous case analysis.

\begin{defn}
    The \emph{transversal basis graph} $\cG_M$ of an antisymmetric matroid $M$ is a graph such that its vertex set is the set of transversal bases of $M$, and two vertices $B$ and $B'$ are adjacent if and only if
    \begin{itemize}
        \item $|B\setminus B'| = 1$ or 
        \item $|B\setminus B'| = 2$ and there is an almost-transversal basis $A$ such that \[ |B\setminus A| = |B'\setminus A| = 1. \]
    \end{itemize} 
    
\end{defn}

The transversal basis graph $\cG_M$ is close to the reduced basis graph $G'_{\Delta(M)}$. 

\begin{lem}\label{lem: compare basis graphs}
Let $M$ be an antisymmetric matroid. Then its transversal basis graph $\cG_M$ can be obtained from $G'_{\Delta(M)}$ by adding some chords to the $4$-cycles of $G'_{\Delta(M)}$ (cf. Lemma~\ref{lem: chord}).  
\end{lem}
\begin{proof}
Let \[
        h(B) := B\cup ([n]\setminus B)^*
        \quad \text{with} \quad
        B \in \Delta(\cB).
    \]
 Then $h$ is a bijection from $\Delta(\cB) = V(G'_{\Delta(M)})$ to $\cB = V(\cG_M)$. Moreover, we have \[|B\triangle B'| = |h(B) \setminus h(B')|.\]

To compare the edges of the two graphs, we have three arguments:
\begin{itemize}
    \item The edges $BB'$ in $G'_{\Delta(M)}$ with $|B\triangle B'| = 1$ are in bijection with the edges $h(B)h(B')$ in $\cG_M$ with $|h(B)\setminus h(B')|=1$.
    \item For any edge $BB'$ in $G'_{\Delta(M)}$ with $|B\triangle B'| = 2$, we show that $h(B)h(B')$ is an edge of $\cG_M$ (with $|h(B) \setminus h(B')|=2$). By definition, there does not exist a pair of distinct vertices $B_1$ and $B_2$ in $G'_{\Delta(M)}$ such that 
        \[
            |B\triangle B_1|=|B\triangle B_2| =|B'\triangle B_1|=|B'\triangle B_2|= 1.\tag{\ref{eq: reduce}}
        \]
Let $B\triangle B' = \{x,y\}$. Then the condition implies that $B\triangle \{x\}$ or $B\triangle \{y\}$ is not a basis of $\Delta(M)$. Let $T = h(B)$, $p=\{x,x^*\}$, and $q=\{y,y^*\}$.
Then $h(B') = T \triangle (p\cup q)$, and $T \triangle p$ or $T \triangle q$ is not a basis of $M$. Therefore, by Lemma~\ref{lem: 3-term Plucker for antisymmetric}, an almost-transversal $A=(T\cup p)\setminus q$ is a basis of $M$.
Because $|h(B_1)\setminus A| = |h(B_2)\setminus A| = 1$, $h(B)h(B')$ is an edge of $\cG_M$.
\item For any edge $h(B)h(B')$ in $\cG_M$ with $|h(B)\setminus h(B')|=2$, if $BB'$ is not an edge of $G'_{\Delta(M)}$, then by definition there exists a pair of distinct vertices $B_1$ and $B_2$ with property \eqref{eq: reduce}. Hence $BB_1B'B_2B$ is a $4$-cycle in $G'_{\Delta(M)}$. \qedhere
\end{itemize}
\end{proof}

As a direct consequence of Theorem~\ref{thm: 2 homotopy theorems for delta-matroids}(i), Lemma~\ref{lem: chord}, and Lemma~\ref{lem: compare basis graphs}, we obtain the homotopy theorem for antisymmetric matroids. 

\begin{thm}\label{thm: Maurer homotopy theorem for antisymmetric matroids}
    Any two walks with the same end-points in the transversal basis graph $\cG_M$ of an antisymmetric matroid $M$ are combinatorially homotopic.
\end{thm}

\section{Homotopy theorem for integral polymatroids}\label{sec: polymatroid}

In this section, we generalize Maurer's homotopy theorem to integral polymatroids.
Let $\Delta_n^r := \{\alpha \in \bZ_{\ge0}^n \colon \sum_{i=1}^n \alpha_i = r\}$ be a dilated discrete simplex.
For $\alpha \in \bZ^n$, let $|\alpha| := |\alpha_1| + |\alpha_2| + \cdots + |\alpha_n|$.

\begin{defn}
    An \emph{integral polymatroid} on $[n]$ is a pair $([n], \cB)$ such that $\cB$ is a non-empty subset of a discrete simplex $\Delta_n^r$ satisfying that for all $\alpha,\alpha'\in \cB$ and $i\in[n]$ with $\alpha_i>\alpha'_i$, there is $j\in[n]$ with $\alpha_j<\alpha'_j$ such that $\alpha-\bfe_i+\bfe_j\in \cB$.
\end{defn}

For simplicity, we call an integral polymatroid a \emph{polymatroid} in the rest of this paper.
Note that a matroid is identified with a polymatroid whose bases are 0/1-vectors. 

\begin{defn}
    The \emph{basis polytope} $\pP_J$ of a polymatroid $J=([n], \cB)$ is the convex hull of the vectors in $\cB$.
\end{defn}

\begin{thm}[see {\cite[Theorem~4.12]{Murota2003}}]
    \label{thm: integer points of polymatroid polytope}
    Let $J = ([n],\cB)$ be a polymatroid. Then $\cB = \pP_J \cap \bZ^n$.
\end{thm}

\begin{thm}[see {\cite[(4.43)]{Murota2003}}]
    \label{thm:GGMS for polymatroids}
    Let $\pP$ be an integral polytope in $\mathbb{R}_{\ge0}^n$. Then $\pP$ is the basis polytope of a polymatroid if and only if every edge is parallel to $\bfe_i-\bfe_j$ for some $i\ne j$.
\end{thm}

\begin{defn}
    The \emph{basis graph} $G_J$ of a polymatroid $J=([n], \cB)$ is a graph on $\cB$ such that two vertices $\alpha$ and $\alpha'$ are adjacent if and only if $|\alpha-\alpha'|=2$.
\end{defn}

We prove the homotopy theorem for polymatroids (Theorem~\ref{thm: Maurer homotopy theorem for polymatroids}) by constructing a polyhedral complex such that the 1-skeleton is the basis graph of a polymatroid and each face is a translate of the basis polytope of a matroid.

The following lemma shows that the basis polytope of a polymatroid can be decomposed into translates of basis polytopes of matroids. It is a standard result in the theory of submodular functions; see~{\cite[\S2]{FM2022}} or {\cite[Lemma~5.19]{Sanchez2023}}. Since no proof is explicitly given there, we provide a proof.

\begin{lem}\label{lem: decomposition of polymatroid polytope}
    Let $J$ be a polymatroid on $[n]$ and $\pQ=\{x \in \mathbb{R}^n \colon 0\le x_i\le 1 \text{ for all } i\}$ be the $n$-dimensional unit cube.
    For any $\alpha\in\bZ_{\ge0}^n$, 
    the intersection $\pP_J \cap (\pQ+\alpha)$ equals a translate of the basis polytope of a matroid, unless the intersection is empty.
\end{lem}
\begin{proof}
    Denote $\pP := \pP_J - \alpha$.
    We claim that $\pP \cap \pQ$ is a basis polytope of a matroid.

    First, we prove that $\pP \cap \pQ$ is a 0/1-polytope. 
    For the sake of contradiction, assume $\pP \cap \pQ$ has a vertex $\beta$ such that $0 < \beta_i < 1$ for some $i$. 
    Let $F$ be the face of $\pP$ whose relative interior contains $\beta$. By Theorem~\ref{thm:GGMS for polymatroids}, $F$ is a translate of a basis polytope of a polymatroid, and hence $F$ is contained in an affine space that is parallel to a linear space generated by vectors of the form $\bfe_j-\bfe_{j'}$. 
    
    Consider an auxiliary graph on $[n]$ such that $j$ and $j'$ are adjacent if $F$ has an edge that is a translate of $\bfe_j-\bfe_{j'}$, and let $S_1, \dots, S_m$ be the vertex sets of the connected components of the auxiliary graph.
    Let $s(j) := \min S_j$.
    Then the linear space parallel to the affine span of $F$ is generated by the $\sum_j ( |S_j|-1 )$ vectors $\bfe_k - \bfe_{s(j)}$ with $1\le j\le m$ and $k \in S_j\setminus\{s(j)\}$.
Because the vertices of $F$ are integral, for any point $\gamma$ in $F$, we have $\sum_{k\in S_j} \gamma_k \in \bZ$ for each $j$.
    Thus, $\sum_{k\in S_j} \beta_k \in \bZ$ for each $j$.
    
    Since $0 < \beta_i < 1$,
    there is $i'$ such that $i$ and $i'$ are in the same $S_j$ and $0 < \beta_{i'} < 1$. Then $F$ contains the line segment between $\beta-\varepsilon (\bfe_i-\bfe_{i'})$ and $\beta+\varepsilon (\bfe_i-\bfe_{i'})$ for sufficiently small $\varepsilon > 0$, and so does $\pP$.
    By taking $\varepsilon < \min\{\beta_i,1-\beta_i,\beta_{i'},1-\beta_{i'}\}$, the line segment is also contained in $\pQ$, contradicting that $\beta$ is a vertex of $\pP \cap \pQ$.
    Therefore, $\pP \cap \pQ$ is a 0/1-polytope.

    We then show that the vertex set $\cB$ of $\pP \cap \pQ$ satisfies the basis exchange property. Let $\beta, \beta' \in \cB$ and $i\in[n]$ with $\beta_i > \beta'_i$. Then there is $j\in[n]$ with $\beta_j < \beta'_j$ such that $\beta-\bfe_i+\bfe_j \in \pP$ because $\pP$ is a translate of the basis polytope of $J$. Since $\beta,\beta' \in \pQ$, we have $\beta_i = 1$ and $\beta_j = 0$. Therefore, $\beta-\bfe_i+\bfe_j$ is in $\pP \cap \pQ$ and hence is in $\cB$.
    This implies that $\pP \cap \pQ$ is the basis polytope of a matroid $([n],\cB)$.
    \qedhere

\end{proof}

\begin{ex}
As we shall see in the proof of Theorem~\ref{thm: Maurer homotopy theorem for polymatroids}, by intersecting a polymatroid polytope with translates of the unit cube, we make the polymatroid a polyhedral complex; 
see Figures~\ref{fig: polymatroid decomposition1} and~\ref{fig: polymatroid decomposition2} for two examples. 
\end{ex}

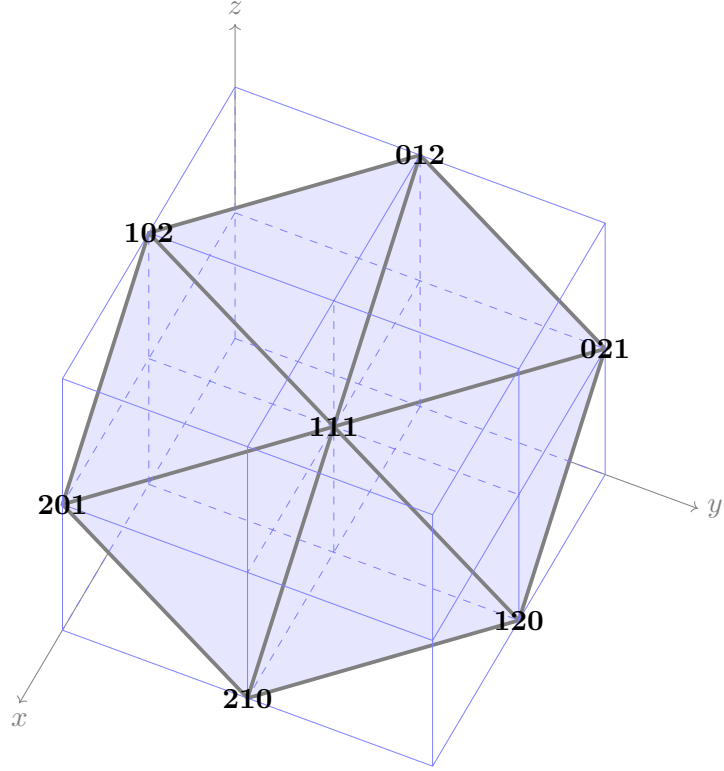
\begin{figure}
    \centering 
    \tdplotsetmaincoords{38}{115}
    \begin{tikzpicture}
        [tdplot_main_coords,
                front/.style={line width=0.5mm,gray},
                back/.style={line width=0.5mm,dashed,gray},
                subdivfront/.style={very thin,blue!50},
                subdivback/.style={very thin,dashed,blue!50},
                axis/.style={->,very thin,gray},
            ]

        \begin{scope}[scale=2.7]
\draw[axis] (-0.0,0,0) -- (2.5,0,0) node[below] {$x$};
            \draw[axis] (0,-0.0,0) -- (0,2.5,0) node[right] {$y$};
            \draw[axis] (0,0,-0.0) -- (0,0,2.5) node[above] {$z$};

\draw[front, fill=blue!10]
                (2,1,0)--(1,2,0)--(0,2,1)--(0,1,2)--(1,0,2)--(2,0,1)--cycle
                ;

            \draw[front]
                (2,1,0)--(0,1,2)
                (1,2,0)--(1,0,2)
                (0,2,1)--(2,0,1)
                ;

\draw[subdivback] 
(0,0,0)--(2,0,0)
                (0,1,0)--(2,1,0)
(0,0,1)--(2,0,1)
                (0,1,1)--(2,1,1)
(0,0,0)--(0,2,0)
                (1,0,0)--(1,2,0)
(0,0,1)--(0,2,1)
                (1,0,1)--(1,2,1)
(0,0,0)--(0,0,2)
                (1,0,0)--(1,0,2)
(0,1,0)--(0,1,2)
                (1,1,0)--(1,1,2)
;

            \draw[subdivfront]
(0,2,0)--(2,2,0)
                (0,2,1)--(2,2,1)
                (0,0,2)--(2,0,2)
                (0,1,2)--(2,1,2)
                (0,2,2)--(2,2,2)
(2,0,0)--(2,2,0)
                (2,0,1)--(2,2,1)
                (0,0,2)--(0,2,2)
                (1,0,2)--(1,2,2)
                (2,0,2)--(2,2,2)
(2,0,0)--(2,0,2)
                (2,1,0)--(2,1,2)
                (0,2,0)--(0,2,2)
                (1,2,0)--(1,2,2)
                (2,2,0)--(2,2,2)
                ;

            \node at (2,1,0) {\bf  210};
            \node at (1,2,0) {\bf  120};
            \node at (0,2,1) {\bf  021};
            \node at (0,1,2) {\bf  012};
            \node at (1,0,2) {\bf  102};
            \node at (2,0,1) {\bf  201};
            \node at (1,1,1) {\bf  111};

        \end{scope}
    \end{tikzpicture}
    \caption{The basis polytope $\pP_J$ of the polymatroid
    \[J=\{(1,1,1)+\bfe_i-\bfe_j\colon i,j\in[3]\}\]
    is a hexagon in the hyperplane $x+y+z=3$. 
    Intersecting $\pP_J$ with integer translates of the unit cube $[0,1]^3$ decomposes the hexagon into six triangles. Three of these triangles are translates of the basis polytope of $U_{1,3}$, while the other three are translates of the basis polytope of $U_{2,3}$. }
    \label{fig: polymatroid decomposition1}
\end{figure}

\begin{figure}
    \centering
    \tdplotsetmaincoords{55}{90}
    \begin{tikzpicture}
            [tdplot_main_coords,
                front/.style={thick,gray},
                back/.style={thick,dashed,gray},
                subdivfront/.style={very thin,blue},
                subdivback/.style={very thin,dashed,blue},
                axis/.style={->,very thin,gray},
            ]

        \begin{scope}[scale=3.0]

\coordinate (2222) at (0,0,0);
            
\coordinate (2011) at (1,0,0);
            \coordinate (0211) at (-1,0,0);

\coordinate (1120) at (0,1,0);
            \coordinate (1102) at (0,-1,0);

\coordinate (2110) at (1/2,1/2,1/1.4142);
            \coordinate (0112) at (-1/2,-1/2,-1/1.4142);

\coordinate (2101) at (1/2,-1/2,1/1.4142);
            \coordinate (0121) at (-1/2,1/2,-1/1.4142);

\coordinate (1210) at (-1/2,1/2,1/1.4142);
            \coordinate (1012) at (1/2,-1/2,-1/1.4142);

\coordinate (1201) at (-1/2,-1/2,1/1.4142);
            \coordinate (1021) at (1/2,1/2,-1/1.4142);

            \fill[color=blue!23] 
                (2011)--(2110)--(1120)--(1021)--cycle
                ;
            \fill[color=blue!14] 
                (1102)--(2101)--(2011)--(1012)--cycle
                ;
            \fill[color=blue!10] 
                (2101)--(2110)--(1210)--(1201)--cycle
                ;
            \fill[color=blue!12] 
                (2011)--(2110)--(2101)--cycle
                ;
            \fill[color=blue!20] 
                (1120)--(1210)--(2110)--cycle
                ;
            \fill[color=blue!5] 
                (1102)--(2101)--(1201)--cycle
                ;
            \fill[color=blue!26] 
                (2011)--(1021)--(1012)--cycle
                ;
            
\draw[front]
                (2011)--(2110)--(1120)--(1021)--cycle
(1102)--(2101)--(2011)--(1012)--cycle
                (2101)--(2110)--(1210)--(1201)--cycle
;

            \draw[front]
                (1012)--(1021)
                (1120)--(1210)
                (1201)--(1102)
                ;

            \draw[back]
                (1120)--(1210)--(0211)--(0121)--cycle
                (0211)--(1201)--(1102)--(0112)--cycle
                (1012)--(1021)--(0121)--(0112)--cycle
                ;

            \node at (2011) {\bf \small 2011};
            \node at (0211) {\bf \small 0211};
            \node at (1120) {\bf \small 1120};
            \node at (1102) {\bf \small 1102};
            \node at (2110) {\bf \small 2110};
            \node at (0112) {\bf \small 0112};
            \node at (2101) {\bf \small 2101};
            \node at (0121) {\bf \small 0121};
            \node at (1210) {\bf \small 1210};
            \node at (1012) {\bf \small 1012};
            \node at (1201) {\bf \small 1201};
            \node at (1021) {\bf \small 1021};

            \node at (0,0,-1.30) {$P_J$};
        \end{scope}

    \end{tikzpicture}

    \vspace{0.3cm}

    \begin{tikzpicture}
        [tdplot_main_coords,
            front/.style={thick,gray},
            back/.style={thick,dashed,gray},
            subdivfront/.style={very thin,blue},
            subdivback/.style={very thin,dashed,blue},
            axis/.style={->,very thin,gray},
        ]

\begin{scope}[scale=2.0, xshift=4.1cm, yshift=0cm]
\begin{scope}[]
\coordinate (2222) at (0,0,0);
                
\coordinate (2011) at (1,0,0);
                \coordinate (0211) at (-1,0,0);

\coordinate (1120) at (0,1,0);
                \coordinate (1102) at (0,-1,0);

\coordinate (2110) at (1/2,1/2,1/1.4142);
                \coordinate (0112) at (-1/2,-1/2,-1/1.4142);

\coordinate (2101) at (1/2,-1/2,1/1.4142);
                \coordinate (0121) at (-1/2,1/2,-1/1.4142);

\coordinate (1210) at (-1/2,1/2,1/1.4142);
                \coordinate (1012) at (1/2,-1/2,-1/1.4142);

\coordinate (1201) at (-1/2,-1/2,1/1.4142);
                \coordinate (1021) at (1/2,1/2,-1/1.4142);

                \fill[blue!5]
                    (2222)--(0211)--(0112)--cycle
                    ;
                \fill[blue!20]
                    (2222)--(0121)--(0211)--cycle
                    ;
                \fill[blue!26]
                    (2222)--(0121)--(0112)--cycle
                    ;

                \draw[front] 
                    (0211)--(0112)--(0121)--cycle
                    ;

                \draw[front]
                    (2222)--(0211)
                    (2222)--(0112)
                    (2222)--(0121)
                    ;
            \end{scope}

\begin{scope}[yshift=-0.72cm]
\coordinate (2222) at (0,0,0);
                
\coordinate (2011) at (1,0,0);
                \coordinate (0211) at (-1,0,0);

\coordinate (1120) at (0,1,0);
                \coordinate (1102) at (0,-1,0);

\coordinate (2110) at (1/2,1/2,1/1.4142);
                \coordinate (0112) at (-1/2,-1/2,-1/1.4142);

\coordinate (2101) at (1/2,-1/2,1/1.4142);
                \coordinate (0121) at (-1/2,1/2,-1/1.4142);

\coordinate (1210) at (-1/2,1/2,1/1.4142);
                \coordinate (1012) at (1/2,-1/2,-1/1.4142);

\coordinate (1201) at (-1/2,-1/2,1/1.4142);
                \coordinate (1021) at (1/2,1/2,-1/1.4142);

                \fill[color=blue!12]
                    (2222)--(1012)--(1021)--cycle
                    ;
                \fill[color=blue!5]
                    (2222)--(0112)--(1012)--cycle
                    ;
                \fill[color=blue!20]
                    (2222)--(1021)--(0121)--cycle
                    ;

                \draw[front] 
                    (1012)--(1021)--(0121)
                    (0112)--(1012)
                    ;
                \draw[back]
                    (0121)--(0112)
                    ;

                \draw[front]
                    (2222)--(1012)
                    (2222)--(1021)
                    (2222)--(0121)
                    (2222)--(0112)
                    ;
            \end{scope}

\begin{scope}[xshift=0.62354cm, yshift=0.36cm]
\coordinate (2222) at (0,0,0);
                
\coordinate (2011) at (1,0,0);
                \coordinate (0211) at (-1,0,0);

\coordinate (1120) at (0,1,0);
                \coordinate (1102) at (0,-1,0);

\coordinate (2110) at (1/2,1/2,1/1.4142);
                \coordinate (0112) at (-1/2,-1/2,-1/1.4142);

\coordinate (2101) at (1/2,-1/2,1/1.4142);
                \coordinate (0121) at (-1/2,1/2,-1/1.4142);

\coordinate (1210) at (-1/2,1/2,1/1.4142);
                \coordinate (1012) at (1/2,-1/2,-1/1.4142);

\coordinate (1201) at (-1/2,-1/2,1/1.4142);
                \coordinate (1021) at (1/2,1/2,-1/1.4142);

                \fill[color=blue!12]
                    (2222)--(1120)--(1210)--cycle
                    ;
                \fill[color=blue!5]
                    (2222)--(1210)--(0211)--cycle
                    ;
                \fill[color=blue!26]
                    (2222)--(0121)--(1120)--cycle
                    ;

                \draw[front] 
                    (1120)--(1210)--(0211)
                    (0121)--(1120)
                    ;
                \draw[back]
                    (0211)--(0121)
                    ;

                \draw[front]
                    (2222)--(1120)
                    (2222)--(1210)
                    (2222)--(0211)
                    (2222)--(0121)
                    ;
            \end{scope}

\begin{scope}[xshift=-0.62354cm, yshift=0.36cm]
\coordinate (2222) at (0,0,0);
                
\coordinate (2011) at (1,0,0);
                \coordinate (0211) at (-1,0,0);

\coordinate (1120) at (0,1,0);
                \coordinate (1102) at (0,-1,0);

\coordinate (2110) at (1/2,1/2,1/1.4142);
                \coordinate (0112) at (-1/2,-1/2,-1/1.4142);

\coordinate (2101) at (1/2,-1/2,1/1.4142);
                \coordinate (0121) at (-1/2,1/2,-1/1.4142);

\coordinate (1210) at (-1/2,1/2,1/1.4142);
                \coordinate (1012) at (1/2,-1/2,-1/1.4142);

\coordinate (1201) at (-1/2,-1/2,1/1.4142);
                \coordinate (1021) at (1/2,1/2,-1/1.4142);

                \fill[color=blue!12]
                    (2222)--(1201)--(1102)--cycle
                    ;
                \fill[color=blue!20]
                    (2222)--(1201)--(0211)--cycle
                    ;
                \fill[color=blue!26]
                    (2222)--(1102)--(0112)--cycle
                    ;

                \draw[front] 
                    (0211)--(1201)--(1102)--(0112)
                    ;
                \draw[back]
                    (0112)--(0211)
                    ;

                \draw[front]
(2222)--(0211)
                    (2222)--(1201)
                    (2222)--(1102)
                    (2222)--(0112)
                    ;
            \end{scope}

\begin{scope}[yshift=0.72cm]
\coordinate (2222) at (0,0,0);
                
\coordinate (2011) at (1,0,0);
                \coordinate (0211) at (-1,0,0);

\coordinate (1120) at (0,1,0);
                \coordinate (1102) at (0,-1,0);

\coordinate (2110) at (1/2,1/2,1/1.4142);
                \coordinate (0112) at (-1/2,-1/2,-1/1.4142);

\coordinate (2101) at (1/2,-1/2,1/1.4142);
                \coordinate (0121) at (-1/2,1/2,-1/1.4142);

\coordinate (1210) at (-1/2,1/2,1/1.4142);
                \coordinate (1012) at (1/2,-1/2,-1/1.4142);

\coordinate (1201) at (-1/2,-1/2,1/1.4142);
                \coordinate (1021) at (1/2,1/2,-1/1.4142);

                \draw[front, fill=blue!12]
                    (2222)--(1210)--(1201)--cycle
                    ;

                \draw[back]
                    (0211)--(2222)
                    (0211)--(1210)
                    (0211)--(1201)
                    ;

\end{scope}

\begin{scope}[xshift=0.62354cm, yshift=-0.36cm]
\coordinate (2222) at (0,0,0);
                
\coordinate (2011) at (1,0,0);
                \coordinate (0211) at (-1,0,0);

\coordinate (1120) at (0,1,0);
                \coordinate (1102) at (0,-1,0);

\coordinate (2110) at (1/2,1/2,1/1.4142);
                \coordinate (0112) at (-1/2,-1/2,-1/1.4142);

\coordinate (2101) at (1/2,-1/2,1/1.4142);
                \coordinate (0121) at (-1/2,1/2,-1/1.4142);

\coordinate (1210) at (-1/2,1/2,1/1.4142);
                \coordinate (1012) at (1/2,-1/2,-1/1.4142);

\coordinate (1201) at (-1/2,-1/2,1/1.4142);
                \coordinate (1021) at (1/2,1/2,-1/1.4142);

                \draw[front, fill=blue!12]
                    (2222)--(1021)--(1120)--cycle
                    ;

                \draw[back]
                    (0121)--(2222)
                    (0121)--(1021)
                    (0121)--(1120)
                    ;

\end{scope}

\begin{scope}[xshift=-0.62354cm, yshift=-0.36cm]
\coordinate (2222) at (0,0,0);
                
\coordinate (2011) at (1,0,0);
                \coordinate (0211) at (-1,0,0);

\coordinate (1120) at (0,1,0);
                \coordinate (1102) at (0,-1,0);

\coordinate (2110) at (1/2,1/2,1/1.4142);
                \coordinate (0112) at (-1/2,-1/2,-1/1.4142);

\coordinate (2101) at (1/2,-1/2,1/1.4142);
                \coordinate (0121) at (-1/2,1/2,-1/1.4142);

\coordinate (1210) at (-1/2,1/2,1/1.4142);
                \coordinate (1012) at (1/2,-1/2,-1/1.4142);

\coordinate (1201) at (-1/2,-1/2,1/1.4142);
                \coordinate (1021) at (1/2,1/2,-1/1.4142);

                \draw[front, fill=blue!12]
                    (2222)--(1102)--(1012)--cycle
                    ;

                \draw[back]
                    (0112)--(2222)
                    (0112)--(1102)
                    (0112)--(1012)
                    ;

\end{scope}

            \node at (0,0,-2.25) {Cells behind};
        \end{scope}

\begin{scope}[scale=2.0, xshift=0.0cm, yshift=-0.0cm]
\begin{scope}[xshift=0.62354cm, yshift=-0.36cm]
\coordinate (2222) at (0,0,0);
                
\coordinate (2011) at (1,0,0);
                \coordinate (0211) at (-1,0,0);

\coordinate (1120) at (0,1,0);
                \coordinate (1102) at (0,-1,0);

\coordinate (2110) at (1/2,1/2,1/1.4142);
                \coordinate (0112) at (-1/2,-1/2,-1/1.4142);

\coordinate (2101) at (1/2,-1/2,1/1.4142);
                \coordinate (0121) at (-1/2,1/2,-1/1.4142);

\coordinate (1210) at (-1/2,1/2,1/1.4142);
                \coordinate (1012) at (1/2,-1/2,-1/1.4142);

\coordinate (1201) at (-1/2,-1/2,1/1.4142);
                \coordinate (1021) at (1/2,1/2,-1/1.4142);

                \draw[front, fill=blue!23] 
                    (2011)--(2110)--(1120)--(1021)--cycle
                    ;
                \draw[front, fill=blue!5] 
                    (2222)--(2011)--(2110)--cycle
                    ;

                \draw[front]
                    (2222)--(2011)
                    (2222)--(2110)
                    ;
                \draw[back]
                    (2222)--(1120)
                    (2222)--(1021)
                    ;
            \end{scope}

\begin{scope}[xshift=-0.62354cm, yshift=-0.36cm]
\coordinate (2222) at (0,0,0);
                
\coordinate (2011) at (1,0,0);
                \coordinate (0211) at (-1,0,0);

\coordinate (1120) at (0,1,0);
                \coordinate (1102) at (0,-1,0);

\coordinate (2110) at (1/2,1/2,1/1.4142);
                \coordinate (0112) at (-1/2,-1/2,-1/1.4142);

\coordinate (2101) at (1/2,-1/2,1/1.4142);
                \coordinate (0121) at (-1/2,1/2,-1/1.4142);

\coordinate (1210) at (-1/2,1/2,1/1.4142);
                \coordinate (1012) at (1/2,-1/2,-1/1.4142);

\coordinate (1201) at (-1/2,-1/2,1/1.4142);
                \coordinate (1021) at (1/2,1/2,-1/1.4142);

                \draw[front, fill=blue!14] 
                    (1102)--(2101)--(2011)--(1012)--cycle
                    ;
                \draw[front, fill=blue!20] 
                    (2222)--(2101)--(2011)--cycle
                    ;

                \draw[front]
(2222)--(2101)
                    (2222)--(2011)
;
                \draw[back]
                    (2222)--(1102)
                    (2222)--(1012)
                    ;
            \end{scope}

\begin{scope}[yshift=0.72cm]
\coordinate (2222) at (0,0,0);
                
\coordinate (2011) at (1,0,0);
                \coordinate (0211) at (-1,0,0);

\coordinate (1120) at (0,1,0);
                \coordinate (1102) at (0,-1,0);

\coordinate (2110) at (1/2,1/2,1/1.4142);
                \coordinate (0112) at (-1/2,-1/2,-1/1.4142);

\coordinate (2101) at (1/2,-1/2,1/1.4142);
                \coordinate (0121) at (-1/2,1/2,-1/1.4142);

\coordinate (1210) at (-1/2,1/2,1/1.4142);
                \coordinate (1012) at (1/2,-1/2,-1/1.4142);

\coordinate (1201) at (-1/2,-1/2,1/1.4142);
                \coordinate (1021) at (1/2,1/2,-1/1.4142);

                \draw[front, fill=blue!10] 
                    (2101)--(2110)--(1210)--(1201)--cycle
                    ;
                \draw[front, fill=blue!26] 
                    (2222)--(2110)--(2101)--cycle
                    ;

                \draw[front]
                    (2222)--(2101)
                    (2222)--(2110)
                    ;
                \draw[back]
                    (2222)--(1210)
                    (2222)--(1201)
                    ;
            \end{scope}

\begin{scope}[xshift=0.62354cm, yshift=0.36cm]
\coordinate (2222) at (0,0,0);
                
\coordinate (2011) at (1,0,0);
                \coordinate (0211) at (-1,0,0);

\coordinate (1120) at (0,1,0);
                \coordinate (1102) at (0,-1,0);

\coordinate (2110) at (1/2,1/2,1/1.4142);
                \coordinate (0112) at (-1/2,-1/2,-1/1.4142);

\coordinate (2101) at (1/2,-1/2,1/1.4142);
                \coordinate (0121) at (-1/2,1/2,-1/1.4142);

\coordinate (1210) at (-1/2,1/2,1/1.4142);
                \coordinate (1012) at (1/2,-1/2,-1/1.4142);

\coordinate (1201) at (-1/2,-1/2,1/1.4142);
                \coordinate (1021) at (1/2,1/2,-1/1.4142);

                \draw[front, fill=blue!20] 
                    (1120)--(1210)--(2110)--cycle
                    ;
                \draw[front, fill=blue!5] 
                    (2222)--(2110)--(1210)--cycle
                    ;
                \draw[front, fill=blue!26] 
                    (2222)--(2110)--(1120)--cycle
                    ;

                \draw[front]
                    (2222)--(1120)
                    (2222)--(1210)
                    (2222)--(2110)
                    ;
            \end{scope}

\begin{scope}[xshift=-0.62354cm, yshift=0.36cm]
\coordinate (2222) at (0,0,0);
                
\coordinate (2011) at (1,0,0);
                \coordinate (0211) at (-1,0,0);

\coordinate (1120) at (0,1,0);
                \coordinate (1102) at (0,-1,0);

\coordinate (2110) at (1/2,1/2,1/1.4142);
                \coordinate (0112) at (-1/2,-1/2,-1/1.4142);

\coordinate (2101) at (1/2,-1/2,1/1.4142);
                \coordinate (0121) at (-1/2,1/2,-1/1.4142);

\coordinate (1210) at (-1/2,1/2,1/1.4142);
                \coordinate (1012) at (1/2,-1/2,-1/1.4142);

\coordinate (1201) at (-1/2,-1/2,1/1.4142);
                \coordinate (1021) at (1/2,1/2,-1/1.4142);

                \draw[front, fill=blue!5] 
                    (1102)--(2101)--(1201)--cycle
                    ;
                \draw[front, fill=blue!20] 
                    (2222)--(2101)--(1201)--cycle
                    ;
                \draw[front, fill=blue!26] 
                    (2222)--(2101)--(1102)--cycle
                    ;

                \draw[front]
                    (2222)--(1102)
                    (2222)--(2101)
                    (2222)--(1201)
                    ;
            \end{scope}

\begin{scope}[yshift=-0.72cm]
\coordinate (2222) at (0,0,0);
                
\coordinate (2011) at (1,0,0);
                \coordinate (0211) at (-1,0,0);

\coordinate (1120) at (0,1,0);
                \coordinate (1102) at (0,-1,0);

\coordinate (2110) at (1/2,1/2,1/1.4142);
                \coordinate (0112) at (-1/2,-1/2,-1/1.4142);

\coordinate (2101) at (1/2,-1/2,1/1.4142);
                \coordinate (0121) at (-1/2,1/2,-1/1.4142);

\coordinate (1210) at (-1/2,1/2,1/1.4142);
                \coordinate (1012) at (1/2,-1/2,-1/1.4142);

\coordinate (1201) at (-1/2,-1/2,1/1.4142);
                \coordinate (1021) at (1/2,1/2,-1/1.4142);

                \draw[front, fill=blue!26] 
                    (2011)--(1021)--(1012)--cycle
                    ;
                \draw[front, fill=blue!5] 
                    (2222)--(1012)--(2011)--cycle
                    ;
                \draw[front, fill=blue!20] 
                    (2222)--(1021)--(2011)--cycle
                    ;

                \draw[front]
                    (2222)--(2011)
                    (2222)--(1021)
                    (2222)--(1012)
                    ;
            \end{scope}

\begin{scope}[]
\coordinate (2222) at (0,0,0);
                
\coordinate (2011) at (1,0,0);
                \coordinate (0211) at (-1,0,0);

\coordinate (1120) at (0,1,0);
                \coordinate (1102) at (0,-1,0);

\coordinate (2110) at (1/2,1/2,1/1.4142);
                \coordinate (0112) at (-1/2,-1/2,-1/1.4142);

\coordinate (2101) at (1/2,-1/2,1/1.4142);
                \coordinate (0121) at (-1/2,1/2,-1/1.4142);

\coordinate (1210) at (-1/2,1/2,1/1.4142);
                \coordinate (1012) at (1/2,-1/2,-1/1.4142);

\coordinate (1201) at (-1/2,-1/2,1/1.4142);
                \coordinate (1021) at (1/2,1/2,-1/1.4142);

                \draw[front, fill=blue!12] 
                    (2011)--(2110)--(2101)--cycle
                    ;

                \draw[back]
                    (2222)--(2011)
                    (2222)--(2110)
                    (2222)--(2101)
                    ;
            \end{scope}

            \node at (0,0,-2.25) {Cells in front};
        \end{scope}
    \end{tikzpicture}
    \caption{The basis polytope $\pP_J$ of the polymatroid$$J = \{(1,1,1,1) + \bfe_i - \bfe_j \colon i,j\in[4]\}.$$Intersecting $\pP_J$ with integer translates of $[0,1]^4$ decomposes the polytope into eight tetrahedra and six pyramids. The tetrahedra are translates of the basis polytopes of $U_{1,4}$ and $U_{3,4}$, and the pyramids are translates of the basis polytope of $U_{2,4}^-$ (where $U_{2,4}^-$ is obtained from $U_{2,4}$ by removing a basis).} 
    \label{fig: polymatroid decomposition2}
\end{figure}
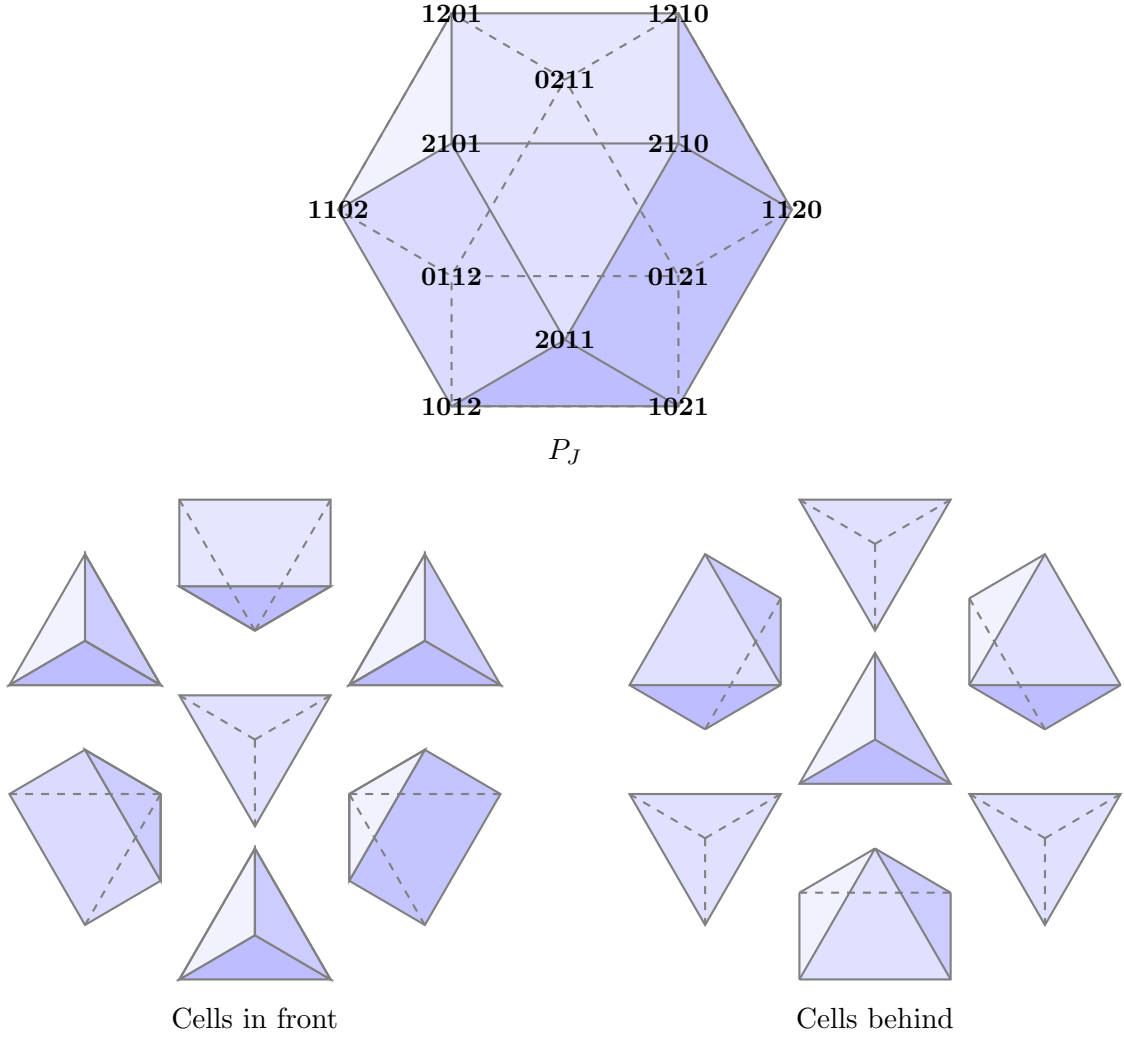

The following lemma describes a simple criterion for a collection of polytopes to form a polyhedral complex.

\begin{lem}\label{lem: simple description of polyhedral complex}
    Let $\pP_1,\cdots,\pP_m$ be polytopes in $\bR^n$ and 
let $\pC$ be the set of all faces of $\pP_1,\cdots,\pP_m$. 
    If $\pP_i \cap \pP_j$ is a common face of $\pP_i$ and $\pP_j$ for any $i\ne j$, then $\pC$ is a polyhedral complex.
\end{lem}
\begin{proof}
    Let $F_1$ and $F_2$ be in $\pC$. 
    To prove that $\pC$ is a polyhedral complex, it suffices to show that $F_1 \cap F_2$ is also in $\pC$ whenever $F_1 \cap F_2 \ne \emptyset$.

    By definition, $F_1$ and $F_2$ are faces of $\pP_i$ and $\pP_j$ for some $i,j\in[m]$, respectively.
    Let $\pP = \pP_i \cap \pP_j$. Then $F_1$ and $\pP$ are faces of $\pP_i$, and so is $F_1 \cap \pP$. Hence, $F_1 \cap \pP$ is a face of $\pP$. Similarly, $F_2 \cap \pP$ is a face of $\pP$.
    Therefore, $F_1 \cap F_2 = F_1\cap F_2 \cap \pP$ is a face of $\pP$, implying that $F_1 \cap F_2 \in \pC$.
\end{proof}

\begin{thm}[Theorem~\ref{thm: polymatroid homotopy theorem}]\label{thm: Maurer homotopy theorem for polymatroids}
    Any two walks with the same end-points in the basis graph $G_J$ of a polymatroid $J$ are combinatorially homotopic.
\end{thm}
\begin{proof}
Consider the non-empty polytopes $\pR_{\alpha} = \pP_J \cap (\pQ+\alpha)$, where $\alpha\in\bZ_{\ge0}^n$.
Let $\pC$ be the set of all faces of such polytopes. 

We first show that $\pC$ is a polyhedral complex.
By Lemma~\ref{lem: simple description of polyhedral complex}, it suffices to show that $\pR_\alpha \cap \pR_{\alpha'}$ is a common face of $\pR_\alpha$ and $\pR_{\alpha'}$ for any $\alpha\ne \alpha'$.
Note that $\pR_\alpha \cap \pR_{\alpha'} = \pR_\alpha \cap (\pQ+\alpha')$ is the intersection of $\pR_i$ and hyperplanes $\{x_k=t_k\}$ for some $k\in[n]$ and $t_i\in\bZ$. Then one can easily take a linear functional showing that $\pR_\alpha\cap\pR_{\alpha'}$ is a face of $\pR_\alpha$. By symmetry, $\pR_\alpha\cap\pR_{\alpha'}$ is also a face of $\pR_{\alpha'}$ and therefore $\pC$ is a polyhedral complex.

We next show that the graph of $\pC$ is $G_J$. 
By definition, the vertices of $G_J$ are exactly the bases of $J$. The vertices of $\pC$ equal the integral points of $\pP_J$, which are exactly the bases of $J$ by Theorem~\ref{thm: integer points of polymatroid polytope}. Thus, the vertex sets of $G_J$ and the graph of $\pC$ are the same.

Suppose that $\beta$ and $\beta'$ are adjacent vertices in $G_J$. Then $\beta-\beta' = \bfe_i-\bfe_j$ for some $i\ne j$. Therefore, there exists an integral translate of the unit cube $[0,1]^n$ containing both $\beta$ and $\beta'$, and hence they are in the same $\pR_\alpha$. By Lemma~\ref{lem: 1-skeleton}, the line segment between $\beta$ and $\beta'$ is a $1$-face of $\pR_{i}$ and hence a $1$-face of $\pC$. 

It remains to show that every $1$-face $L$ of $\pC$ corresponds to an edge of $G_J$. By definition, $L$ is a $1$-face of $\pR_{i}$ for some $i$. Note that $\pR_i$ is a translate of a matroid basis polytope by Lemma~\ref{lem: decomposition of polymatroid polytope}. 
Hence, the $1$-face $L$ is a translate of $\bfe_i-\bfe_j$, which corresponds to an edge $G_J$.
We conclude that $G_J$ is the graph of $\pC$.

The $2$-faces of $\pC$ are the $2$-faces of $\pR_{\alpha}$'s by definition and hence are triangles and squares by Lemma~\ref{lem: 2-face}. By Proposition~\ref{prop: main2}, we obtain the desired result. 
\end{proof}

\section{Complements}

\subsection{A short proof of the Gelfand--Serganova theorem for delta-matroids}\label{sec: GGMS for delta-matroids}

\begin{thm}[Theorem~\ref{thm:GGMS for delta-matroids}(i)]
    Let $\pP$ be a 0/1-polytope. Then $\pP$ is a basis polytope of a delta-matroid if and only if every edge is a translate of a vector of the form $\bfe_i$ for some $i\in [n]$ or of the form $\bfe_i-\bfe_j$ or $\bfe_i+\bfe_j$ for two distinct elements $i,j\in [n]$. 
\end{thm}

The following proof is a direct extension of that of Theorem~4.1 for matroids in~\cite{GGMS1987}. 

\begin{proof}

Let $\cB := \{B \subseteq [n] \colon \bfe_B \text{ is a vertex of } \pP\}$. Then $\pP = \pP_\cB$.

Suppose that every edge of $\pP$ satisfies the condition stated in the theorem. We need to prove that for $A,B\in \cB$ and $i \in A\triangle B$, there is an element $j\in A\triangle B$ (possibly, $j=i$) such that $A\triangle \{i,j\} \in \cB$.

By replacing with $\cB$ with $\cB\triangle A = \{B'\triangle A\colon A\in \cB\}$, we may assume that $A=\emptyset$, as $\pP_{\cB\triangle A}$ can be obtained from $\pP_\cB$ by reflecting along the hyperplanes $x_k = \frac{1}{2}$ for $k\in A$, which has the same edge property. We need to prove that for $i\in B$, there is an element $j\in B$ such that $\{i,j\} \in \cB$.

    Since $\pP$ is convex, the line segment between $\bfe_A = \mathbf{0}$ and $\bfe_B$ is contained in $\pP$. Hence, it lies in a cone spanned by the edges of $\pP$ emanating from $\mathbf{0}$. We denote these edge vectors by $\bfv_1, \bfv_2, \ldots, \bfv_m$ and write 
    \[
        \bfe_B = \lambda_1 \bfv_1 + \lambda_2 \bfv_2 + \cdots + \lambda_m \bfv_m,
    \]
    where the $\lambda_k$'s are nonnegative reals. Since $\pP$ is a 0/1-polytope,  every $\bfv_k$ equals $\bfe_s$ for some $s\in [n]$ or $\bfe_s+\bfe_t$ for two distinct elements $s,t\in [n]$. For the element $i\in B$, since $\bfe_B(i) = 1 > 0$, there exists a vector $\bfv_k$ equal to either $\bfe_i$ or $\bfe_i + \bfe_j$, which implies that either $\{i\}\in\cB$ or $\{i,j\}\in\cB$.

    Conversely, suppose $\cB$ is the collection of bases of a delta-matroid. Any edge of $\pP$ is of the form $\{(1-\lambda)\bfe_A + \lambda \bfe_B\colon0\leq\lambda\leq 1\}$ for distinct $A,B\in \cB$. When $|A\triangle B|<3$, the edge has the desired property. When $|A\triangle B| \ge 3$, we need to show that this cannot be an edge. It is enough to show that there exists $0<\lambda<1$ such that $(1-\lambda)\bfe_A + \lambda \bfe_B$ is a convex combination of the vertices of $\pP$ other than $\bfe_A$ and $\bfe_B$.

    By replacing with $\cB$ with $\cB\triangle A = \{B'\triangle A : A\in \cB\}$, we may assume that $A=\emptyset$. 
    By relabeling, we may assume that $B = \{1,2,\ldots,k\}$ for some $3\le k\le n$.

    We construct an auxiliary directed graph $G$ on $\{1,2,\ldots,k\} \sqcup \{1',2',\ldots,k'\}$, where the edge set is
    \[
        \{\vec{ij'} : \{i,j\} \in \cB\} \sqcup \{\vec{i'j} : B\triangle \{i,j\} \in \cB\}.
    \]
    By the symmetric exchange axiom, every vertex of $G$ has outdegree at least $1$. Hence, $G$ has a directed cycle, and we enumerate its vertices accordingly as $i_0, i_1', i_2, i_3' \dots, i_{2m-1}', i_{2m}$ with $i_{2m} = i_0 \in \{1,2,\ldots,k\}$.
    For $j=0,1,\dots,2m-1$, let
    \[
    \bfv_j := \begin{cases}
        \bfe_{i_j} + \bfe_{i_{j+1}} & \text{if } i_j \ne i_{j+1} \\
        \bfe_{i_j} & \text{if } i_j = i_{j+1} \\
    \end{cases}
    \quad\text{and}\quad
    \sigma_j := \begin{cases}
        1 & \text{if } i_j \ne i_{j+1} \\
        2 & \text{if } i_j = i_{j+1}. \\
    \end{cases}\]

We have 
    \[
         \sum_{j:\text{ even}} \sigma_j \cdot \bfv_{j}    =\sum_{0\leq j\leq 2m-1}  \bfe_{i_j}    = 
\sum_{j:\text{ odd}} \sigma_j \cdot \bfv_{j},     \]
which implies
\[
        \left( \sum_{j:\text{ odd}} \sigma_j \right) \bfe_B 
        = 
        \left( \sum_{j:\text{ even}} \sigma_j \cdot \bfv_{j} \right) 
        +
        \left( \sum_{j:\text{ odd}} \sigma_j \cdot (\bfe_{B} - \bfv_{j} ) \right).
    \]

Because the support of $\bfv_j$ is a basis in $\cB\setminus\{A,B\}$ for each even $j$, and $\bfe_B - \bfv_j$ is a basis in $\cB\setminus\{A,B\}$ for each odd $j$, we obtain the desired convex combination by taking $\lambda = (\sum_{j:\text{ odd}} \sigma_j) / (\sum_{j=0}^{2m-1} \sigma_j)$.
\end{proof}

Lastly, we use the Gelfand--Serganova Theorem for delta-matroids to derive the one for even delta-matroids and the Gelfand--Goresky--MacPherson--Serganova Theorem for matroids.

\begin{proof}[Proof of Theorem~\ref{thm:GGMS for delta-matroids}(ii)]
    Recall that a delta-matroid is even if and only if all its bases have the same parity. In the Gelfand--Serganova Theorem for delta-matroids, vectors of the form $\bfe_i$ do not exist if and only if all the bases have the same parity.  
\end{proof}

\begin{proof}[Proof of Theorem~\ref{thm:GGMS}]
    Recall that a delta-matroid is a matroid if and only if all its bases have the same cardinality. In the Gelfand--Serganova Theorem for delta-matroids, vectors of the form $\bfe_i$ or $\bfe_i+\bfe_j$ do not exist if and only if all the bases have the same cardinality.  
\end{proof}
\subsection{A homology theory}\label{sec: homology}
In \cite[Theorem~1.12]{Wenzel1996}, Wenzel reformulated
Maurer's homotopy theorem (Theorem~\ref{thm: Maurer homotopy theorem}) in terms of homology. We slightly change his wording to avoid defining unnecessary terminology. 
\begin{thm}[{\cite[Theorem~1.12]{Wenzel1996}}]\label{thm: Wenzel homology theorem}
The first homology group\footnote{It makes sense to speak of the homology of a graph by viewing the graph as a CW-complex; see Def.~\ref{def: homology}.} of the basis graph of an even delta-matroid is generated by cycles of length $3$ and $4$.  
\end{thm}
However, we argue that it is not obvious that the homology statement is equivalent to the homotopy statement. It is true that the latter implies the former. 

\begin{prop}\label{prop: homotopy implies homology}
Let $G$ be a connected simple graph. If any two walks with the same end-points are
combinatorially homotopic in $G$, then the first homology group of $G$ is generated by cycles of length $3$ and $4$. 
\end{prop}
\begin{proof}
Let $C$ be a cycle, and we need to prove that it can be written as a sum of cycles of length $3$ and $4$ in the homology group. By the homotopy property, the cycle $C$ (viewed as a closed walk) can be turned into a vertex by elementary transformations. Notice that if two cycles $C_1$ and $C_2$ differ by an elementary transformation, then they differ by a $3$-cycle or a $4$-cycle in the homology group ($2$-cycles are zero in the homology group). Hence we get the desired result. 
\end{proof}

We will exhibit a connected simple graph for which the homology statement holds but the homotopy statement does not. In this sense, the two statements are not obviously equivalent, although they are both true when the graph is the basis graph of a matroid or a delta-matroid.  

Before showing the example, we give some comments on the literature concerning the homology statement. Although Wenzel used a homology statement, his proof was for the homotopy statement. We can still say that Wenzel generalized Maurer’s result in this sense. 
As we mentioned in Section~\ref{sec: intro}, an important application of Theorem~\ref{thm: Maurer homotopy theorem} is to prove the cryptomorphisms for matroids over tracts in \cite{BB}\footnote{ For the part of proof that requires Maurer's homotopy theorem, the authors of~\cite{BB} refer to a similar proof in the paper \cite{AD2012} by Anderson and Delucchi, where Maurer's homotopy theorem is used.}. But in these proofs, we only need the homology property. One can see this clearly in \cite{JK} where Jin and the second author use Wenzel's homology statement to prove the cryptomorphisms for orthogonal matroids (equivalent to even delta-matroids) over tracts. In \cite[Theorem~5.3]{Kim}, the second author shows a homology property of the transversal basis graph of an antisymmetric matroid and uses it to prove the cryptomorphisms. This homology property of antisymmetric matroids only generalizes Wenzel's homology statement for even delta-matroids. 
Our homotopy theorem for antisymmetric matroids (Theorem~\ref{thm: Maurer homotopy theorem for antisymmetric matroids}) together with Proposition~\ref{prop: homotopy implies homology} implies \cite[Theorem~5.3]{Kim} directly.

The remaining part of this subsection is to prove Proposition~\ref{prop: counterexample}.

We briefly recall the cellular homology. For the undefined terms, we refer the readers to Hatcher's book \cite{Hatcher} or Kozlov's book \cite{Kozlov}. A CW complex is said to be \emph{regular} if the attaching maps are all homeomorphisms. A polytope is naturally a regular CW complex, where faces correspond to closed cells. Hence a polytope of dimension $d$ admits a cellular chain complex
\[0\xrightarrow{} A_d\xrightarrow{\partial_d} \cdots \to A_1\xrightarrow{\partial_1} A_0\to 0,\] 
where $A_i$ is the free abelian group with basis the $i$-faces of $X$. The chain complex has the following properties, where the latter two are due to regularity.  
\begin{itemize}
    \item $\partial_{i-1}\circ\partial_i=0$ for each $i$.
    \item For any $i$-face $f$ with $i\geq 2$, we have\[\partial_if=\sum_\text{$e$ is an $(i-1)$-face of $f$}k_ee,\]where each $k_e$ is $1$ or $-1$, depending on the orientation of $e$. 
    \item For any edge $e$, we have\[\partial_1e=\pm(v-v'),\]where $v$ and $v'$ are the two vertices incident to $e$.   
\end{itemize}

Recall that a graph $G$ can also be viewed as a CW complex $X_G$. When the graph is loopless, the CW complex is regular. 

\begin{defn}\label{def: homology}
Let $G$ be a connected simple graph. The homology group $H_1(G)$ of the graph is defined to be the first homology group (with coefficients in $\mathbb{Z}$) of the CW complex $X_G$. 
\end{defn}
The cellular chain complex of $X_G$ is
\begin{equation}\label{eq: chain complex X_G}
0\to A_1\xrightarrow{\partial_1} A_0\to 0,    
\end{equation}
where $A_1$ is the free abelian group with basis the edges of $G$, $A_0$ is the free abelian group with basis the vertices of $G$. Hence $H_1(G)=\ker(\partial_1)$. A cycle of $G$ clearly gives rise to an element (up to signs) in $H_1(G)$, and such elements generate $H_1(G)$. By abusing terminology, we say that $H_1(G)$ is generated by the cycles of $G$.

\begin{lem}
Let $G$ be a graph. Then $H_1(G)$ is generated by cycles of length $3$ and $4$ if and only if $H_1(X^{34}_G)$ is trivial (cf. Def.~\ref{def: X_G}).  
\end{lem}
\begin{proof}
The space $X^{34}_G$ is a regular CW complex, and its cellular chain complex is
\[0\to A_2\xrightarrow{\partial_2} A_1\xrightarrow{\partial_1} A_0\to 0,\]where $A_0$ and $A_1$ are as in \eqref{eq: chain complex X_G}, and $A_2$ is the free abelian group with basis the $2$-cells of $X^{34}_G$. By the definition of $X^{34}_G$, the image of $\partial_2$ is generated by $3$- and $4$-cycles of $G$. Thus the homology group $H_1(X^{34}_G)=\ker\partial_1/\im \partial_2$ is trivial if and only if $H_1(G)=\ker(\partial_1)$ is generated by $3$- and $4$-cycles of $G$.
\end{proof}

We recall the classical technique in algebraic topology to construct a \(2\)-dimensional CW complex $X$ with $\pi_1(X)\cong\alt(5)$, where \[\alt(5)\cong
\left\langle a,b \mid a^{2}=1,\ b^{3}=1,\ (ab)^{5}=1\right\rangle\] is the alternating group of degree five. 
We begin with the wedge
\(X^{(1)}=S^{1}_{a}\vee S^{1}_{b}\), whose two oriented circles represent
the generators \(a\) and \(b\), and attach three \(2\)-cells along loops
representing the words \(a^{2}\), \(b^{3}\), and \((ab)^{5}\), respectively.
Thus by Van Kampen's theorem,
\[
\pi_{1}(X)
\cong \operatorname{Alt}(5).
\]
Since the homology group is isomorphic to the abelianization of the fundamental group, we obtain that $H_1(X)$ is trivial. 

Since $X$ is not a regular CW complex, there is no graph $G$ with $X=X^{34}_G$. Hence, we need to modify the space $X$. Every CW complex $X$ is homotopy equivalent to a simplicial complex, which can be chosen to be of the same dimension as $X$ and finite if $X$ is finite \cite[Theorem~2C.5]{Hatcher}. Let $Y$ be a $2$-dimensional finite simplicial complex homotopy equivalent to the space $X$ that we constructed. For each $2$-cell of $Y$, which must be a triangle, we subdivide it into three quadrilaterals; see Figure~\ref{fig: subdivide}. By doing so, we obtain a new CW complex $Z$. Then the $1$-skeleton of $Z$ is a graph $G$ with $X^{34}_G=Z$. Since $Z$ is still homotopy equivalent to the space $X$, we obtain the following result.

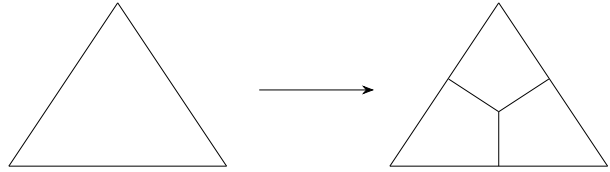
\begin{figure}[!ht]
\centering
\resizebox{0.5\textwidth}{!}{\begin{circuitikz}
\tikzstyle{every node}=[font=\fontsize{18.2pt}{23.7pt}\selectfont]
\draw [short] (13.75,20.375) -- (11.25,16.625);
\draw [short] (13.75,20.375) -- (16.25,16.625);
\draw [short] (11.25,16.625) -- (16.25,16.625);
\draw [short] (22.5,20.375) -- (20,16.625);
\draw [short] (20,16.625) -- (25,16.625);
\draw [short] (22.5,20.375) -- (25,16.625);
\draw [short] (22.5,17.875) -- (22.5,16.625);
\draw [short] (22.5,17.875) -- (23.650,18.625);
\draw [short] (22.5,17.875) -- (21.350,18.625);
\draw [-{Stealth[scale=1.5]}] (17,18.375) -- (19.625,18.375);
\end{circuitikz}
}\caption{A subdivision of a triangle into three quadrilaterals.}
\label{fig: subdivide}
\end{figure}

\begin{prop}\label{prop: counterexample}
There exists a connected simple graph $G$ such that $H_1(G)$ is generated by cycles of length $3$ and $4$ but the graph admits a closed walk that is not combinatorially null-homotopic. 
\end{prop}

\subsection{A concluding remark}\label{sec: high dim}
As mentioned earlier, we only need the homology property of the basis graphs of matroids, orthogonal matroids, and antisymmetric matroids to prove the cryptomorphisms. In the matroid case, the homology property is that any cycle of the basis graph can be generated by cycles of lengths 3 and 4. Moreover, from our proof, it is easy to see that any cycle of the basis graph can be generated by the boundaries of $2$-faces (which are triangles and squares) in the basis polytope. This property can be naturally generalized to higher dimensions as follows. Recall that the basis polytope of dimension $d$ admits a cellular chain complex
\[0\xrightarrow{} A_d\xrightarrow{\partial_d} \cdots \to A_1\xrightarrow{\partial_1} A_0\to 0.\] 
Then the general homology property is that the $k$-cycles (elements in $\ker\partial_k$) are generated by the 
$k$-boundaries (elements in $\im\partial_{k+1}$) because all homology groups of a polytope vanish. For a fixed $k$, the $(k+1)$-faces of matroid polytopes, which give the generators of the $k$-boundaries, fall into finitely many combinatorial types. In this paper, we deal with the case $k=1$, where the $2$-faces are triangles and squares. By computing these types, one can obtain the homology properties of higher dimensions. However, we have not found any applications of the higher-dimensional cases.

\section*{Acknowledgments}
We thank Matthew Baker, Andreas Holmsen, Seunghun Lee, and Fr\'{e}d\'{e}ric Meunier for their valuable comments.
Both authors are partially supported by the AMS-Simons Travel Grant.

\section*{Statement on AI usage}
ChatGPT came up with the idea of the counterexample in Section~\ref{sec: homology}. We used AI to find some references, and we also used AI to check grammar when preparing this manuscript.

\printbibliography

@misc{KP2026,
	author = {Kim, Donggyu and Pomeranz, Ari},
	title = {Circuit functions of polymatroids with coefficients},
	note = {In preparation},}

@misc{BHKL2025b,
	archiveprefix = {arXiv},
	author = {Matthew Baker and June Huh and Mario Kummer and Oliver Lorscheid},
	eprint = {2508.02907},
	primaryclass = {math.CO},
	title = {Lorentzian polynomials and matroids over triangular hyperfields 1: Topological aspects},
	url = {https://arxiv.org/abs/2508.02907},
	year = {2025},}

@misc{BHKKL2025,
	archiveprefix = {arXiv},
	author = {Matthew Baker and June Huh and Donggyu Kim and Mario Kummer and Oliver Lorscheid},
	eprint = {2507.14718},
	primaryclass = {math.CO},
	title = {Representation theory for polymatroids},
	url = {https://arxiv.org/abs/2507.14718},
	year = {2025},}

@article{AD2012,
	author = {Anderson, Laura and Delucchi, Emanuele},
	doi = {10.1007/s00454-012-9458-9},
	fjournal = {Discrete \& Computational Geometry. An International Journal of Mathematics and Computer Science},
	issn = {0179-5376},
	journal = {Discrete Comput. Geom.},
	mrclass = {52B40 (05B35 52C40 94C99)},
	mrnumber = {3000567},
	mrreviewer = {Winfried Hochst\"{a}ttler},
	number = {4},
	pages = {807--846},
	title = {Foundations for a theory of complex matroids},
	url = {https://doi.org/10.1007/s00454-012-9458-9},
	volume = {48},
	year = {2012},}

@article{CCO2015,
	author = {Chalopin, J\'{e}r\'{e}mie and Chepoi, Victor and Osajda, Damian},
	doi = {10.1016/j.jctb.2015.03.004},
	fjournal = {Journal of Combinatorial Theory. Series B},
	issn = {0095-8956},
	journal = {J. Combin. Theory Ser. B},
	mrclass = {05B35},
	mrnumber = {3354288},
	mrreviewer = {Simon Hampe},
	pages = {1--32},
	title = {On two conjectures of {M}aurer concerning basis graphs of matroids},
	url = {https://doi.org/10.1016/j.jctb.2015.03.004},
	volume = {114},
	year = {2015},}

@article{Wenzel1996,
	author = {Wenzel, Walter},
	doi = {10.1006/aama.1996.0002},
	fjournal = {Advances in Applied Mathematics},
	issn = {0196-8858},
	journal = {Adv. in Appl. Math.},
	mrclass = {05B35 (51D20)},
	mrnumber = {1376641},
	mrreviewer = {Andr\'{e} Bouchet},
	number = {1},
	pages = {27--62},
	title = {Maurer's homotopy theory and geometric algebra for even {$\Delta$}-matroids},
	url = {https://doi.org/10.1006/aama.1996.0002},
	volume = {17},
	year = {1996},}

@article{BL2005,
	author = {Barcelo, H\'{e}l{\`e}ne and Laubenbacher, Reinhard},
	doi = {10.1016/j.disc.2004.03.016},
	fjournal = {Discrete Mathematics},
	issn = {0012-365X},
	journal = {Discrete Math.},
	mrclass = {52B40 (05B35 05E25 37F20 52C35 55R80 57Q05)},
	mrnumber = {2163440},
	mrreviewer = {Jean-Louis Cathelineau},
	number = {1-3},
	pages = {39--61},
	title = {Perspectives on {$A$}-homotopy theory and its applications},
	url = {https://doi.org/10.1016/j.disc.2004.03.016},
	volume = {298},
	year = {2005},}

@article{BKL2001,
	author = {Barcelo, H\'{e}l{\`e}ne and Kramer, Xenia and Laubenbacher, Reinhard and Weaver, Christopher},
	doi = {10.1006/aama.2000.0710},
	fjournal = {Advances in Applied Mathematics},
	issn = {0196-8858,1090-2074},
	journal = {Adv. in Appl. Math.},
	mrclass = {57Q05 (05B35 05C10 55P99 55Q05)},
	mrnumber = {1808443},
	mrreviewer = {Andrew\ Vince},
	number = {2},
	pages = {97--128},
	title = {Foundations of a connectivity theory for simplicial complexes},
	url = {https://doi.org/10.1006/aama.2000.0710},
	volume = {26},
	year = {2001},}

@misc{Sanchez2023,
	archiveprefix = {arXiv},
	author = {Mario Sanchez},
	eprint = {2311.04203},
	primaryclass = {math.AG},
	title = {Derived Categories of Permutahedral and Stellahedral Varieties},
	url = {https://arxiv.org/abs/2311.04203},
	year = {2023}}

@article{FM2022,
	author = {Frank, Andr\'{a}s and Murota, Kazuo},
	doi = {10.1007/s10107-021-01722-2},
	fjournal = {Mathematical Programming},
	issn = {0025-5610,1436-4646},
	journal = {Math. Program.},
	mrclass = {90C27 (05C20)},
	mrnumber = {4499076},
	number = {1-2, Ser. A},
	pages = {977--1025},
	title = {Decreasing minimization on {M}-convex sets: background and structures},
	url = {https://doi.org/10.1007/s10107-021-01722-2},
	volume = {195},
	year = {2022},}

@book{Murota2003,
	author = {Murota, Kazuo},
	doi = {10.1137/1.9780898718508},
	isbn = {0-89871-540-7},
	mrclass = {90-02 (52-02 90C27 90C46 91B02)},
	mrnumber = {1997998},
	mrreviewer = {Ulrich Faigle},
	pages = {xxii+389},
	publisher = {Society for Industrial and Applied Mathematics (SIAM), Philadelphia, PA},
	series = {SIAM Monographs on Discrete Mathematics and Applications},
	title = {Discrete convex analysis},
	url = {https://doi.org/10.1137/1.9780898718508},
	year = {2003},}

@article{GS1987,
	author = {Gelfand, Israel M. and Serganova, Vera V.},
	fjournal = {Akademiya Nauk SSSR i Moskovskoe Matematicheskoe Obshchestvo. Uspekhi Matematicheskikh Nauk},
	issn = {0042-1316},
	journal = {Uspekhi Mat. Nauk},
	mrclass = {32M10 (22E40 32C42)},
	mrnumber = {898623},
	mrreviewer = {Gerhard Pfister},
	number = {2(254)},
	pages = {107--134, 287},
	title = {Combinatorial geometries and the strata of a torus on homogeneous compact manifolds},
	url = {https://mathscinet.ams.org/mathscinet-getitem?mr=898623},
	volume = {42},
	year = {1987},}

@article{GGMS1987,
	author = {Gelfand, Israel M. and Goresky, Mark and MacPherson, Robert D. and Serganova, Vera V.},
	doi = {10.1016/0001-8708(87)90059-4},
	fjournal = {Advances in Mathematics},
	issn = {0001-8708},
	journal = {Adv. Math.},
	mrclass = {14M15 (05B35 22E45 22E70 32C38 32C45 32M10)},
	mrnumber = {877789},
	mrreviewer = {Hiroaki Terao},
	number = {3},
	pages = {301--316},
	title = {Combinatorial geometries, convex polyhedra, and {S}chubert cells},
	url = {https://doi.org/10.1016/0001-8708(87)90059-4},
	volume = {63},
	year = {1987},}

@book{Ziegler1995,
	author = {Ziegler, G\"{u}nter M.},
	doi = {10.1007/978-1-4613-8431-1},
	isbn = {0-387-94365-X},
	mrclass = {52Bxx},
	mrnumber = {1311028},
	mrreviewer = {Margaret M. Bayer},
	pages = {x+370},
	publisher = {Springer-Verlag, New York},
	series = {Graduate Texts in Mathematics},
	title = {Lectures on polytopes},
	url = {https://doi.org/10.1007/978-1-4613-8431-1},
	volume = {152},
	year = {1995},}

@article{BGW1997,
	author = {Borovik, Alexandre V. and Gelfand, Israel M. and White, Neil},
	doi = {10.1007/BF02558470},
	fjournal = {Annals of Combinatorics},
	issn = {0218-0006},
	journal = {Ann. Comb.},
	mrclass = {05B35 (05E15 52B40)},
	mrnumber = {1629677},
	number = {2},
	pages = {123--134},
	title = {Coxeter matroid polytopes},
	url = {https://doi.org/10.1007/BF02558470},
	volume = {1},
	year = {1997},}

@article{Maurer1973,
	author = {Maurer, Stephen B.},
	doi = {10.1016/0095-8956(73)90005-1},
	fjournal = {Journal of Combinatorial Theory. Series B},
	issn = {0095-8956},
	journal = {J. Combin. Theory Ser. B},
	mrclass = {05B35},
	mrnumber = {317971},
	mrreviewer = {J. A. Bondy},
	pages = {216--240},
	title = {Matroid basis graphs. {I}},
	url = {https://doi.org/10.1016/0095-8956(73)90005-1},
	volume = {14},
	year = {1973},}

@article {BB,
    AUTHOR = {Baker, Matthew and Bowler, Nathan},
     TITLE = {Matroids over partial hyperstructures},
   JOURNAL = {Adv. Math.},
  FJOURNAL = {Advances in Mathematics},
    VOLUME = {343},
      YEAR = {2019},
     PAGES = {821--863},
   MRCLASS = {52B40 (05B35 12K99 52C40)},
  MRNUMBER = {3891757},
MRREVIEWER = {Dillon\ Mayhew},
       DOI = {10.1016/j.aim.2018.12.004},
       URL = {https://doi.org/10.1016/j.aim.2018.12.004},
}

@article {Kim,
    AUTHOR = {Kim, Donggyu},
     TITLE = {Baker--{B}owler theory for {L}agrangian {G}rassmannians},
   JOURNAL = {Int. Math. Res. Not. IMRN},
  FJOURNAL = {International Mathematics Research Notices. IMRN},
      YEAR = {2025},
    NUMBER = {8},
     PAGES = {Paper No. rnaf094, 41},
   MRCLASS = {14M15 (05B35 14T15 52B40)},
  MRNUMBER = {4892772},
       DOI = {10.1093/imrn/rnaf094},
       URL = {https://doi.org/10.1093/imrn/rnaf094},
}

@article {JK,
    AUTHOR = {Jin, Tong and Kim, Donggyu},
     TITLE = {Orthogonal matroids over tracts},
   JOURNAL = {Forum Math. Sigma},
  FJOURNAL = {Forum of Mathematics. Sigma},
    VOLUME = {13},
      YEAR = {2025},
     PAGES = {Paper No. e130, 34},
   MRCLASS = {05B35 (15A63 52B40)},
  MRNUMBER = {4941601},
       DOI = {10.1017/fms.2025.10085},
       URL = {https://doi.org/10.1017/fms.2025.10085},
}

@book {Hatcher,
    AUTHOR = {Hatcher, Allen},
     TITLE = {Algebraic topology},
 PUBLISHER = {Cambridge University Press, Cambridge},
      YEAR = {2002},
     PAGES = {xii+544},
      ISBN = {0-521-79160-X; 0-521-79540-0},
   MRCLASS = {55-01 (55-00)},
  MRNUMBER = {1867354},
MRREVIEWER = {Donald\ W.\ Kahn},
}

@book {Kozlov,
    AUTHOR = {Kozlov, Dmitry},
     TITLE = {Combinatorial algebraic topology},
    SERIES = {Algorithms and Computation in Mathematics},
    VOLUME = {21},
 PUBLISHER = {Springer, Berlin},
      YEAR = {2008},
     PAGES = {xx+389},
      ISBN = {978-3-540-71961-8},
   MRCLASS = {55-02 (05C15 05C25 06A07 52B70 55U10 57-02)},
  MRNUMBER = {2361455},
MRREVIEWER = {Rade\ \v Zivaljevi\'c},
       DOI = {10.1007/978-3-540-71962-5},
       URL = {https://doi.org/10.1007/978-3-540-71962-5},
}

@book {TGT,
    AUTHOR = {Gross, Jonathan L. and Tucker, Thomas W.},
     TITLE = {Topological graph theory},
    SERIES = {Wiley-Interscience Series in Discrete Mathematics and
              Optimization},
      NOTE = {A Wiley-Interscience Publication},
 PUBLISHER = {John Wiley \& Sons, Inc., New York},
      YEAR = {1987},
     PAGES = {xvi+351},
      ISBN = {0-471-04926-3},
   MRCLASS = {05C10 (20F32 57M15)},
  MRNUMBER = {898434},
MRREVIEWER = {Saul\ Stahl},
}

@article {Lovasz,
    AUTHOR = {Lov\'asz, L\'{a}szl\'{o}},
     TITLE = {A homology theory for spanning trees of a graph},
   JOURNAL = {Acta Math. Acad. Sci. Hungar.},
  FJOURNAL = {Acta Mathematica. Academiae Scientiarum Hungaricae},
    VOLUME = {30},
      YEAR = {1977},
    NUMBER = {3-4},
     PAGES = {241--251},
      ISSN = {0001-5954,1588-2632},
   MRCLASS = {05C10 (05C05 55A15 55B10)},
  MRNUMBER = {543668},
MRREVIEWER = {Dana\ May\ Latch},
       DOI = {10.1007/BF01896190},
       URL = {https://doi.org/10.1007/BF01896190},
}

@book {Spanier,
    AUTHOR = {Spanier, Edwin H.},
     TITLE = {Algebraic topology},
 PUBLISHER = {McGraw-Hill Book Co., New York-Toronto-London},
      YEAR = {1966},
     PAGES = {xiv+528},
   MRCLASS = {55.00},
  MRNUMBER = {210112},
MRREVIEWER = {S.-T.\ Hu},
}

\appendix

\end{document}